%% file: manuscript_v1.tex
\documentclass[twocolumn]{autart}  
\usepackage{graphicx}          
\usepackage{amsmath}
\usepackage{amssymb}
\usepackage{mathtools}
\usepackage{bm}
\usepackage{bbm}
\usepackage{mathrsfs}
\usepackage[sort&compress,square,numbers]{natbib}

\usepackage{xcolor}

\usepackage[normalem]{ulem}

\usepackage{enumerate}
\usepackage{enumitem}

\usepackage{subcaption}

\usepackage{amsfonts}

\usepackage[mathscr]{euscript}

\usepackage{url}

\usepackage{enumerate}

\usepackage{graphicx}
\graphicspath{{./}{./photos}{./figures}}

\usepackage[noadjust]{cite} 

\usepackage{algorithm}
\usepackage{algorithmicx}
\usepackage[noend]{algpseudocode}
\algrenewcommand\algorithmicrequire{\textbf{Input:}}
\algrenewcommand\algorithmicensure{\textbf{Output:}}

\newtheorem{theorem}{Theorem}
\numberwithin{theorem}{section}
\newtheorem{lemma}[theorem]{Lemma}
\newtheorem{proposition}[theorem]{Proposition}
\newtheorem{remark}[theorem]{Remark}

\newtheorem{assumption}[theorem]{Assumption}
\newtheorem{corollary}[theorem]{Corollary}

\newtheorem{problem}[theorem]{Problem}

\let\oldpf\pf
\let\oldendpf\endpf
\def\pf{\begingroup \oldpf}
\def\endpf{\hfill~\qed \oldendpf \endgroup}

\expandafter\let\csname endpf*\endcsname=\endpf

\newenvironment{proof}[1][Proof.]{\begin{pf*}{\MakeUppercase{#1}}}{\end{pf*}}

\newcommand{\mR}{\mathbb{R}}

\newcommand{\mK}{\mathbb{K}}

\DeclareMathOperator*{\st}{s.\hspace{-2pt}t.}

\DeclareMathOperator*{\diag}{diag}

\DeclareMathOperator*{\mE}{\mathbb{E}}
\DeclareMathOperator*{\mP}{\mathbb{P}}

\DeclareMathOperator*{\support}{supp}

\DeclareMathOperator*{\tv}{T.V.}
\DeclareMathOperator*{\lip}{Lip}
\DeclareMathOperator*{\diam}{diam}

\newcommand{\thetatrue}{\bar{\theta}}
\newcommand{\Vtrue}{\bar{V}}

\newcommand{\elltrue}{\bar{\ell}}
\newcommand{\pitrue}{\bar{\pi}}

\newcommand{\mfx}{\bm{x}}

\newcommand{\mfw}{\bm{w}}
\newcommand{\mfu}{\bm{u}}

\newcommand{\mfa}{\bm{a}}
\newcommand{\mfT}{\bm{T}}

\newcommand{\proj}{\mathscr{P}}
\newcommand{\indicator}{\mathbb{I}}
\newcommand{\contfun}{\mathcal{C}}

\allowdisplaybreaks

\edef\endfrontmatter{
  \unexpanded\expandafter{\endfrontmatter}
  \noexpand\endNoHyper 
}

\begin{document}

\begin{frontmatter}
  \title{Consistent inverse optimal control\\ for discrete-time nonlinear stochastic systems}

  \thanks[footnoteinfo]{This work was supported by National Natural Science Foundation (NNSF) of China under Grant 62573286, by Natural Science Foundation of Shanghai under Grant 25ZR1401208, by the Wallenberg AI, Autonomous Systems and Software Program (WASP) funded by the Knut and Alice Wallenberg Foundation, and the Swedish Research Council (VR) under grant 2024-05776.
  }
  \author[SJTU,IMR]{Ziliang Wang}\ead{qingyingyese@sjtu.edu.cn},
  \author[SJTU,IMR]{Han Zhang}\ead{zhanghan\_tc@sjtu.edu.cn},
  \author[CHALMERS_AND_GU]{Axel Ringh}\ead{axelri@chalmers.se}
  
  \address[SJTU]{School of Automation and Intelligent Sensing, Shanghai Jiao Tong University, Shanghai, China}
  \address[IMR]{Institute of Medical Robotics, Shanghai Jiao Tong University, Shanghai, China}
  \address[CHALMERS_AND_GU]{Department of Mathematical Sciences, Chalmers University of Technology and University of Gothenburg, 41296 Gothenburg, Sweden}  


  \begin{abstract}
    
    Inverse Optimal Control (IOC) seeks to recover an unknown cost from expert demonstrations, and it provides a systematic way of modeling experts' decision mechanisms while considering the prior information of the cost functions. Nevertheless, existing IOC methods have consistency issue with the estimator under noisy and nonlinear settings. 
    
    In this paper, we consider a discrete-time nonlinear system with process noise, and it is controlled by an optimal policy that minimizes the expectation of a discounted cumulative cost function across an infinite time-horizon.
    In particular, the cost function takes the form of a linear combination of a priori known feature functions. 
    In this setting, we first adopt Lasserre's reformulation \citep{hernandez2012discrete} of the forward problem with occupancy measure. Next, we propose the infinite dimensional IOC algorithm and further approximate it with Lagrange interpolating polynomials, which results in a convex, finite-dimensional sum-of-squares optimization. 
    Moreover, the estimator is shown to be asymptotically and statistically consistent.
    Finally, we validate the theoretical results and illustrate the performance of our method with numerical experiments. In addition, the robustness and generalizability performance of the proposed IOC algorithm are also illustrated.
  \end{abstract}

\end{frontmatter}

\section{Introduction}\label{sec:introduction}

\input{1-Introduction.tex}

\section{Problem formulation}\label{sec:problem_formulation}

\input{2-ProblemFormulation.tex}

\section{The algorithm construction}\label{sec:alg_construct}
\input{4-IOCAlgorithm.tex}

\section{Consistency analysis}\label{sec:consistency}

\input{5-ConsistenceAnalysis.tex}

\section{On pratical implementation }\label{sec:practical_implementation}
\input{6-PraticalImplement.tex}

\section{Numerical experiments}\label{sec:numerical-expe}
In this section, we evaluate the performance of the proposed method on IOC problems on three typical systems. In particular,
\begin{enumerate}
    \item Linear quadratic regulator (LQR), where the optimal policy and value function admit analytically polynomial representations. 
    \item Temperature control system which is nonlinear and the dynamics is characterized by polynomial. 
    \item Inverted pendulum system which is nonlinear and the dynamics involves sinusoidal terms. 
\end{enumerate}
The experimental results for these three cases are presented  below.

\subsection{Linear quadratic regular}
Consider the linear system
\begin{align*}
    &\mfx_{t+1} =  A\mfx_t + B\mfa_t + \mfw_t,\\
    &A = \left[\begin{matrix}
        1 &0.1\\
        0 &1
    \end{matrix}\right], B = \left[\begin{matrix}
        0.005\\
        0.1
    \end{matrix}\right], \mfw_t\sim \mathcal{N}_{[-1,1]^2}(0,0.01^2I),
\end{align*}
where $I$ is identity matrix. 
The proposed method is evaluated on 400 systems, all of them are with the decay constant $\alpha=0.99$, and then with the cost function that takes the forms
\begin{align*}
    &\ell(x, a) = x^TQx + a^TRa ,\\
    &Q = \text{diag}(q_1, q_2), R = r\\
    &(q_1, q_2, r) = (\bar{q}_1, \bar{q}_2, \bar{r})/\|(\bar{q}_1, \bar{q}_2, \bar{r})\|_2,
\end{align*}
where $(\bar{q}_1, \bar{q}_2, \bar{r})$ is sampled from $\mathcal{U}(0,1]^3$.
Moreover, we set the polynomial degrees to $d_\psi=3$ and $d_V=2$, which is higher than needed to estimate the cost function. The reason for doing this is to assess the robustness of the proposed method under an over-complete basis representation.
More specifically, a total of 128 systems were random initialized by standard Gauss distribution and simulated over 4 steps to collect data.
As shown in Fig.\ref{fig:lqr-hist}, the proposed method accurately recovers the coefficients of cost function in the LQR case, with the estimation errors concentrated near zero. 

\begin{figure}[!htpb]
    \includegraphics[width=\linewidth]{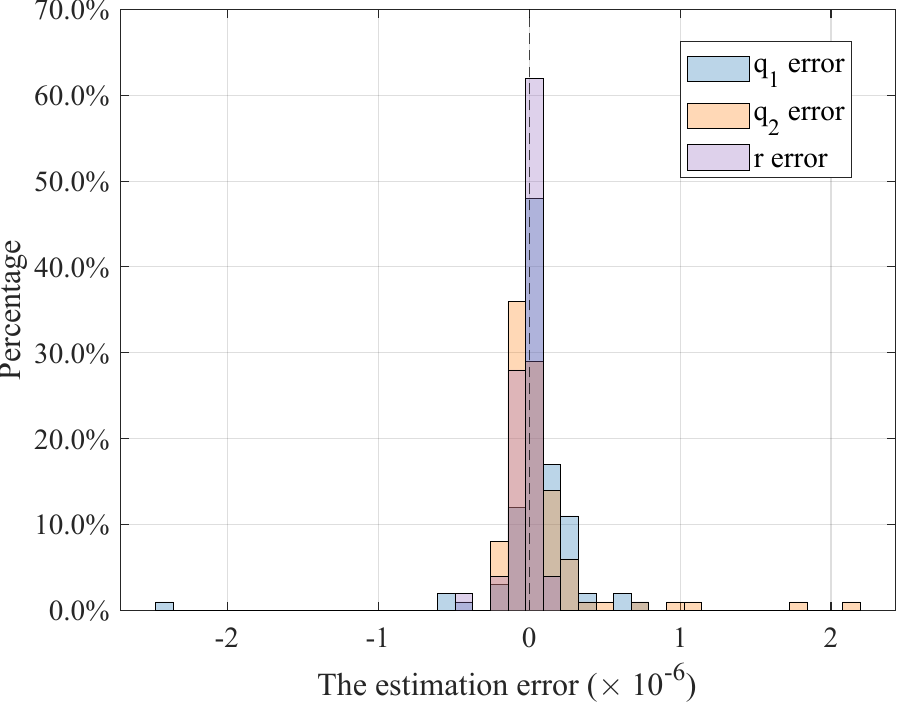}
    \caption{Histogram of Estimation Errors for the LQR Case. Only the central 99.5\% of the samples are displayed. The experiment uses redundant polynomial degrees ($d_\psi=3, d_V=2$) to access the robustness of the proposed method under an overcomplete basis representation.}
    \label{fig:lqr-hist}
\end{figure}

\subsection{Temperature control system}
Consider the temperature control system
\begin{align*}
    &\mfT_{t+1} = \mfT_t + 0.5\mfu - \mathfrak a(\mfT_t-T_{\text{env}}) - \mathfrak b(\mfT_t^4 - T_{\text{env}}^4) + \mfw_t,
\end{align*}
where $\mathfrak a = 5\times10^{-4}, \mathfrak b = 2.268\times10^{-8}, T_{\text{env}} = 290, \mfw_t \in \mathcal{N}_{[-200,200]}(0, 2^2)$.
In particular, the radiative heat dissipation term $\mathfrak b(\mfT_t^4 - T_{\text{env}}^4)$ introduces nonlinearities to the system dynamics. Similar to the LQR case, we construct 100 cost functions with randomly chosen coefficients. In particular,
\begin{align*}
    &\ell(T, a) = q(T/200 - 3.25)^2 + r(a/500)^2,\\
    &(q, r) = (\bar q, \bar r)/\|(\bar{q}, \bar{r})\|_2,
\end{align*}
where $(\bar q, \bar r)$ is sampled from $\mathcal{U}(0,1]^2$.
Moreover, to the best of our knowledge, such nonlinear forward optimal control problems do not have explicit solution for the optimal control policy. Therefore, to numerically generate an expert demonstration dataset,  we employ the Model Predicitve Control (MPC) algorithm with horizen length 64, and minimize the cost function with decay constant $\alpha=0.9$. Using this approach,
for each cost function, the expert demonstration dataset is generated by simulating 32 temperature control systems over 2 time steps. 

We estimated the coefficients of these cost functions under different polynomial approximation degrees and evaluated the estimation accuracy by comparing the normalized estimated coefficients with the ground-truth values. 
As shown in Fig.\ref{fig:tem-hist}, the proposed method performs well for this nonlinear case. Furthermore, higher approximation degrees lead to smaller errors, which is consistent with theoretical expectations.
\begin{figure}[!htpb]
    \includegraphics[width=\linewidth]{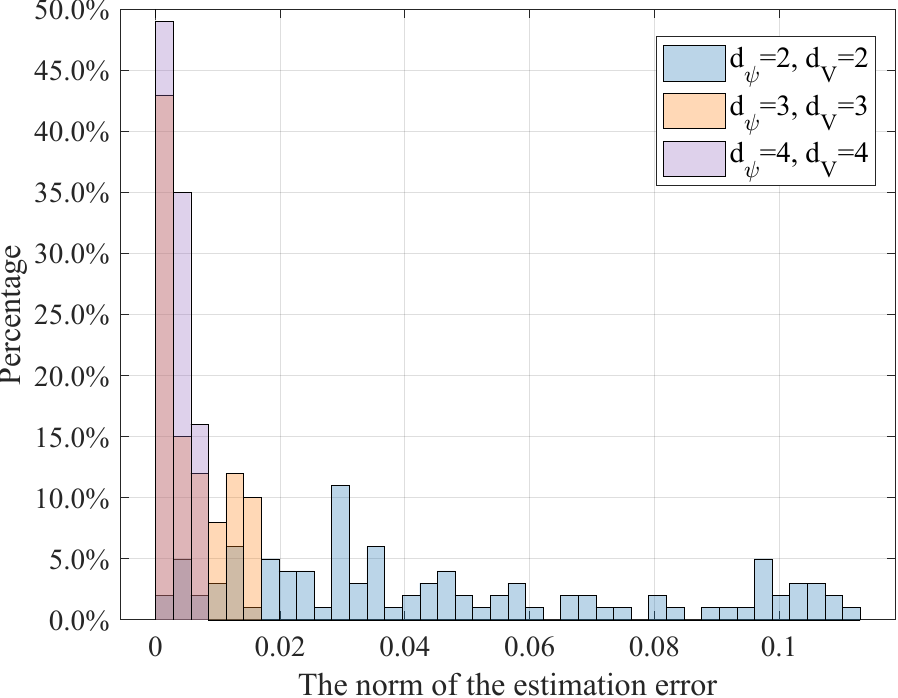}
    \caption{Estimate error histogram for the Temperature control system under different approximation degrees. Only the central 99.5\% of data points are plotted. Estimation accuracy imporves steadily as the approximation degree increases from 2 to 4, with narrower error distributions for higher-order approximations.
    }
    \label{fig:tem-hist}
\end{figure}

\subsection{Inverted pendulum system}
Consider the inverted pendulum system
\begin{equation}
    \begin{aligned}
        &p_{t+1} = p_t + 0.01v_t,\\
        &v_{t+1} = 0.999v_t + 0.01\text{sin}(p_t) + \text{cos}(p_t)u_t.
    \end{aligned}\label{eq:inverted-pendulum-dynamics}
\end{equation}
Notably, the above dynamics is deterministic and only weakly continuous. Hence it does not satisfy our Assumption \ref{ass:transition_kernel} that the transition kernel is Lipschitz continuous in total variation distance. 
Yet we choose to do the experiment under this set-up to show that the algorithm is still able to achieve a good performance despite the fact that the stand-alone assumptions are not fulfilled completely.

Moreover, analogous to the temperature control problem, this problem does not, to the best of our knowledge, have an explicit solution for the optimal control policy.
To numerically generate the expert dataset, we employ an MPC with a 64-step time horizon length and decay constant $\alpha=0.9$ to minimize the running cost function 
\begin{align*}
    \ell(p_t, v_t, u_t) = 100p_t^2 + 10v_t^2 + u_t^2.
\end{align*}

Since the amplitude of the coefficient vectors does not influence the optimal control policy, both the ground truth and estimated coefficient vectors are normalized. Then the distances between these normalized coefficient vectors are computed to quantify the relative estimate errors. As shown in Fig.~\ref{fig:ivp-data-size}, the estimated errors decrease sharply as the data amount increases, and the error stabilizes around $100$ data samples. Further data additions yield no observable improvements, as the inherent precision limit of the approximation itself (rather than data scarcity) becomes the dominant factor restricting error reduction. 
\begin{figure}[!htpb]
    \includegraphics[width=\linewidth]{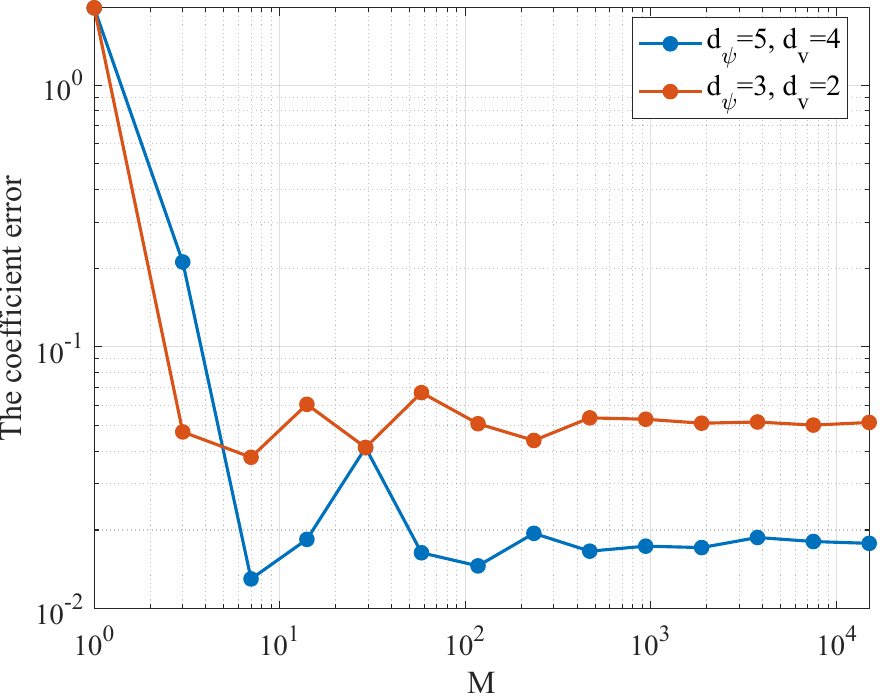}
    \caption{Estimated accuracy versus data volume $M$. The full dataset contains 16384 one-step trajectories. For each value of $M$, the training set is obtained by randomly sampling $M$ trajectories from the full dataset. With approximately $100$ data samples, the estimated errors reach around $5\%$ for $d_\psi=3, d_\psi=2$ and around $2\%$ for $d_\psi=5, d_\psi=4$, after that they level off.}
    \label{fig:ivp-data-size}
\end{figure}

Fig.~\ref{fig:ivp-degrees} shows further results for different approximation degrees.
In these experiments, we also look at the control policy estimation error.
To this end, the expert control policy is recovered using the algorithm explained in Section~\ref{sec:policy-recovery}, where
a four-layer neural network with two hidden layers of sizes 128 and 64 is used to fit the control policy. We generate a dataset of $32768$ trajectories, with random initial value given by $\mathcal{N}_{[-1, 1]^2}(0, \diag(0.314^2,0))$, envolving $4$ time steps.
The dataset is randomly divided into training and testing subsets with a ratio of 4:1. 
The recovery accuracy for the policy is evaluated on the test set using the Symmetric Mean Absolute Percentage Error (SMAPE) between the recovered and the ground truth policies, which is defined as
\begin{align*}
    e_{\mathrm{SMAPE}} = \frac{2}{M_{\text{test}}} \sum_{i=1}^{M_{\text{test}}}\frac{\left|a_{i}-g(x_{i};\theta_\pi)\right|}{\left|a_{i}\right|+\left|g(x_{i};\theta_\pi)\right|}.
\end{align*}
As illustrated in Fig.~\ref{fig:ivp-degrees}, as the approximation degree increases, both the coefficient estimation error and the policy recovery error decrease, which is consistent with our expectation of enhanced accuracy for higher-order approximations.
\begin{figure}[!htpb]
    \includegraphics[width=\linewidth]{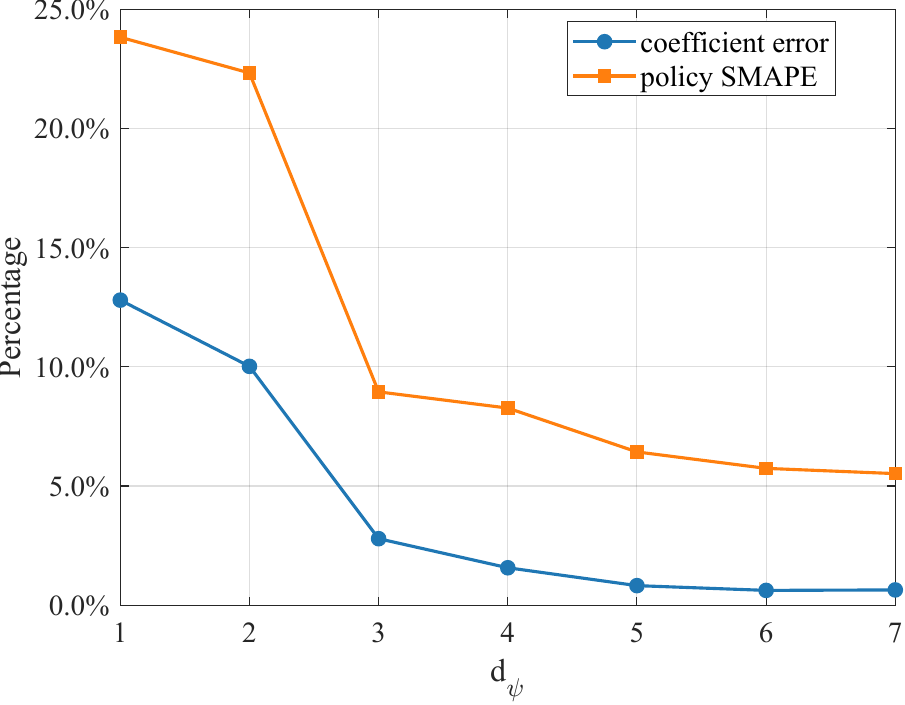}
    \caption{Estimated errors of the cost function coefficient vectors and SMAPE of the recovered control policy across varying approximation degrees. All experiments utilize around $131$k data samples. For each $d_\psi$, best result across different $d_V$ values is reported.}
    \label{fig:ivp-degrees}
\end{figure}

Finally, we also present an additional experiment that shows the robustness and generalizability of the proposed method compared to ``behaviour cloning''. In this setting, the expert data is generated on a new system whose parameters are perturbed, i.e.,
\begin{align*}
    &\bm p_{t+1} = \bm p_t + 0.01\bm v_t + \bm w_{1,t},\\
    &\bm v_{t+1} = 0.999\bm v_t + 0.0105\text{sin}(\bm p_t) + 1.05\text{cos}(\bm p_t)u_t +\bm w_{2,t},
\end{align*}
where $\mfw_{i,t}\sim \mathcal{N}_{[-0.4, 0.4]^2}(0,0.004^2)$ are independent process noise.
In particular, we simulate 16 perturbed systems over 256 steps, and then estimate the coefficients in the cost function by assuming an unchanged system dynamics \eqref{eq:inverted-pendulum-dynamics}. The initial values are randomly sampled according to $(p_0, v_0)\sim \mathcal U[-2.0, 2.0]^2$, and the test system is initialized as $(\bm p_0, v_0) = (3.0, 0.0)$. This simulate the case of out-of-distribution data samples mentioned in Sec. \ref{sec:problem_formulation}.
As shown in Fig.~\ref{fig:ivp-comp}, the proposed method shows a better performance than behavior cloning method.
\begin{figure}[!htpb]
    \includegraphics[width=\linewidth]{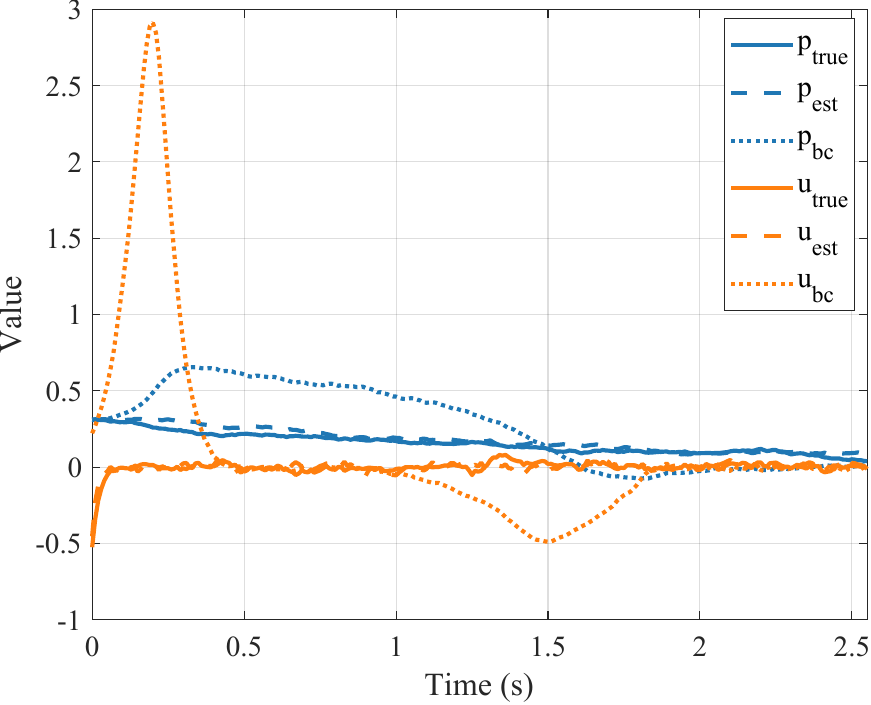}
    \caption{The angle and control trajectories of the test inverted pendulum system. Trajectories is generated by three different controller: (1) MPC controller using the ground truth coefficients (labeled by ``true"); (2) MPC controller using the estimated coefficients (labeled by "est"); (3) Behavior cloning controller (three-layes neural network with 128 hidden units) trained on the same dataset (labeled by ``bc"). }
    \label{fig:ivp-comp}
\end{figure}

\section{Conclusions}\label{sec:conclusions}
\input{8-Conclusion.tex}

\bibliographystyle{plain}
\bibliography{ref}

\end{document}

%% file: 1-Introduction.tex
Optimal Control (OC) provides a principled framework for designing a control law by specifiying it to be the minimizer of a specified performance index. 
However, the performance of such controller in the contextual environment critically depends on the choice of the cost function, which in general trade-offs between the control effort and the feedback system's performance. In practice, accurate specification of such cost function is often nontrivial; it requires expertise and may not reflect the true objectives of the expert's decisions.
Inverse optimal control (IOC), also known as Inverse Reinforcement Learning (IRL), addresses this issue by inferring the underlying cost function of an expert from demonstration data. 
Therefore, it not only provides valuable insights for controller design, but also further enables the reconstruction of the corresponding closed-loop policy \citep{dey2023inverse, mombaur2010human}.

Compared with Imitation Learning (IL), which directly imitates the observed expert behavior \citep{zare2024survey}, IOC assumes that the expert's cost function follows a known structure, such as a linear combination of predefined feature functions \citep{ziebart2008maximum,zhang2019inverse,garrabe2025convex}.
This assumption restricts the optimal policy's possible structure and provides stronger robustness and generalizability.
Motivated by this, IOC methods have been successfully applied to modeling the decision criteria in various domains such as autonomous driving \citep{likmeta2021dealing}, human sensorimotor \citep{Schultheis2021Inverse}, and animal behaviours \citep{Ashwood2022Dynamic}.
However, for nonlinear stochastic systems, the corresponding IOC problems are challenging due to the absence of explicit expressions for the value functions.
This leads to ``coupled-grey-black-box" identification problems, compared to the grey-box case in Linear Quadratic (LQ) settings \citep{zhang2019inverse,zhang2019inverseCDC,li2018convex,zhang2021inverse,zhang2022statistically,yu2021system, Menner2018Convex, rickenbach2024inverse, guo2023imitation, zhang2024statistically, zhang2024inverse}.
Such problem still remains even when the cost function structure is known a priori. 
Moreover, the trajectory observations are usually contaminated by noise in practice. In particular, the noise could be either the observation noise introduced by the sensor measurements, or the process noise that is intrinsically modelled by the Markov transition kernel of the nonlinear stochastic system.
Therefore,  a key property that an IOC estimator should have to guarantee the robustness is consistency.
Nevertheless, if the IOC estimator is an ``M-estimator" \citep[Chap.~5]{van1998asymptotic}, consistency is typically only guaranteed at the globally optimal solution of the underlying optimization problem. Therefore,  it is also desirable that the this optimization is convex, as to avoid issues with local minima and thus achieve consistency in practice.

\subsection{Related works}
One way to avoid the aforementioned ``coupled-grey-black-box" issue is to identify a cost function that satisfies the first-order necessary optimality conditions given by Pontryagin's Maximum Principle (PMP) \citep{Liang2023DataDriven}.
However, such algorithms only recovers a cost function to which the observed trajectories is only locally optimal, and hence we might not be able to guarantee the feedback system's performance when using the control policy recovered from the estimated cost function.
Another straightforward method is to minimize the discrepancy between the expert's policy and the optimal policy induced by the candidate cost functions \citep{abbeel2004apprenticeship, ziebart2008maximum, Ho2016GAIL}.
Such methods often require repeatedly solving the entire or part of the forward optimal control problem, which makes them computationally demanding and less scalable to complex systems.
Nevertheless, by exploiting the specific structures of the system dynamics or the cost function, several studies have avoided the issue of solving the forward problem repeatedly.
In particular, bilinear \citep{fernandez2025estimating}, affine \citep{lian2022inverse, Jean2019Injectivity, Rodrigues2022Inverse}, and polynomial dynamical systems \citep{pauwels2016linear} have been investigated, where certain problem structures enable tractable reformulations. Yet more general nonlinear settings are not considered therein.
Furthermore, another line of research incorporates statistical learning principles.
Inspired by maximum-entropy regularization, a family of methods recovers the cost function by maximum-likelihood estimation under stochastic expert policies \citep{ziebart2008maximum, dvijotham2010inverse, garrabe2025convex, Mehr2023Maximum, Guo2021Deep}.
In addition, \citep{ziebart2008maximum, zhou2018Infinite} discretize the state and control spaces, while others represent the stochastic policies with exponential-family distributions \citep{dvijotham2010inverse, garrabe2025convex, Mehr2023Maximum} or deep neural networks \citep{Guo2021Deep}.
Nevertheless, the impact of these approximations on estimation accuracy remains theoretically unquantified.

\subsection{The contributions}
To address the above challenges, we propose an IOC algorithm for discrete-time infinite-horizon nonlinear stochastic systems with a discounted-cost, where the running cost is assumed to be a linear combination of known continuous feature functions.
In particular, the method recovers the underlying cost function by solving a Sum-Of-Square (SOS) programming problem, which is convex and well-studied. 
Compared to existing approaches, our method
(i) handles the nonlinear and stochastic dynamics without restrictive structural assumptions;
(ii) avoids solving the forward optimal control problem repeatedly; and 
(iii) provides theoretical guarantees of asymptotic and statistical consistency, which have not been established simultaneously in existing works.
More precisely, our key contributions are as follows:
\begin{enumerate}[label=\roman*.]
    \item We reformulate the nonlinear, infinite-horizon IOC problem as an infinite dimensional linear programming problem, which minimizes the violation of optimality conditions for the observed trajectories. We further show that samples of finite-time observation suffice to identify the cost function parameters. By applying a polynomial approximation, we obtain a finite dimensional SOS program that can be solved numerically. Such method also exempts us from local minima issues.
    \item We establish asymptotic and statistical consistency of the proposed estimator. Namely, as the polynomial approximation order and the number of demonstration trajectories increases, the optimality gap between the expert policy and the estimated optimal policy vanishes, and the estimated parameters approach the true ones whenever the solution is unique. Moreover, the underlying optimal control policy can be further consistently recovered, e.g., by integrating the estimated optimality condition into ``behaviour cloning''.
    \item Numerical experiments are conducted to validate the theoretical consistency of the proposed method, demonstrate its estimation accuracy, and evaluate its robustness and generalizability.
\end{enumerate}

The remainder of this paper is organized as follows.
Sec.~\ref{sec:problem_formulation} introduces the problem setting, states several mild assumptions on the considered system, and formulates the IOC problem for discrete-time nonlinear stochastic systems.
Next, we present the proposed approach in Sec.~\ref{sec:alg_construct}. In particular, the unknown functions in the IOC formulation are approximated by polynomials, and the problem is converted into a finite-dimensional SOS program --- this is the finite-dimensional estimator.
Moreover, the theoretical guarantees, including asymptotic and statistical consistency of the estimator, are established in Sec.~\ref{sec:consistency}.
In addition, we briefly describes the SOS optimization procedure used in practice in Sec.~\ref{sec:practical_implementation}. 
In Sec.~\ref{sec:numerical-expe}, we evaluate perform  numerical experiments and the estimator on three benchmark systems: a LQ regulator, a temperature-control process, and an inverted pendulum. In the inverted pendulum problem, we also compare the controllers obtained via IOC with ``behaviour cloning'' (see, e.g., \citep{sammut2017behavioral, codevilla2019exploring, torabi2018behavioralcloningobservation}) and show that IOC gives controllers that are more robust to distributional shifts and errors in the modelled dynamics.
Finally, the paper is concluded in Sec.~\ref{sec:conclusions} with some discussion and outlooks.

\textit{Notation}: 
$\mR_+$ denotes the non-negative real numbers. $C(n,k)$ denotes the binomial coefficient (``$n$ choose $k$").
We denote $\mathcal{M}(A)$ as the space of all finite signed measures on $A$, and let $\support(\mu)$ denote the support of $\mu\in\mathcal{M}(A)$, while $\mathscr{F}(A)$ denotes the space of all bounded measurable functions on $A$. In addition, let $\contfun(A)$ denote the set of continuous functions whose domain is $A$. 
Furthermore, we denote $\|f\|_{p,A}:=\left(\int_A\left|f(x)\right|^p\mathrm{d}x\right)^{\frac{1}{p}}$ as the $p$-norm of the bounded function $f$ on the set $A\subset \mR^n$. In both this and the previous notation, we sometimes drop the set $A$ when there is no risk of confusion.
For vector valued function $f:\mR^n\rightarrow\mR^n$, $\|f\|_{\infty,A}:=\max_{i=1:n}\|f_i\|_{\infty,A}$, where $f_i$ is the $i$-th element of $f$.
Similarly, $\|v\|_p:=\left(\sum_{i=1}^n|v_i|^p\right)^{\frac{1}{p}}$ for $v=[v_1,\ldots,v_n]^T\in\mR^n$.
On the other hand, we denote $\|\mu\|_{\tv}:=\sup_{\|v\|_\infty\leq 1}\int v(x) \mu(\mathrm{d}x)$ as the total variation norm.
We also denote $\lip(\beta)\subset \mathcal{C}(A)$ as the set of all Lipschitz continuous functions on $A$ that has the Lipschitz constant $\beta$.
We further denote $\mR[x]$ as the ring of multivariate polynomials of $x\in\mR^n$ and $\Sigma^2$ as the SOS polynomials. More specifically, $p\in\Sigma^2$ if there exists an integer $m$ and $u_i\in\mR[x]$, such that $p=\sum_{i=1}^m u_i^2$.
Moreover, we define the Cartesian product as $\times_{i=1}^n A_i := A_1\times \ldots \times A_n$.
$\langle\cdot,\cdot\rangle$ denotes the dual bracket between the dual pair $(\cdot,\cdot)$.
Moreover, for linear operator $\mathscr{G}$, we let $\mathscr{G}^*$ denote its adjoint.
For any class $B\subset \mR^n$, the diameter is defined as $\diam(B):=\sup_{x,y\in B}\|x-y\|_2$.
In addition, $\indicator_{\Gamma}(x)$ denotes the indicator function of $x$ on set $\Gamma$, namely,
\[
\indicator_\Gamma(x) = 
\begin{cases}
1 &\mbox{if } x\in\Gamma,\\
0 &\mbox{otherwise}.
\end{cases}
\]
Moreover, we use $\textbf{\textit{italic bold font}}$ to denote stochastic elements, and $\mfx \sim \mu$ means that $\mfx$ has distribution $\mu$. Specifically, $\mathcal{N}(b, \sigma^2)$ denotes normal distribution with mean $b$ and variance $\sigma^2$ and $\mathcal{N}_C(b, \sigma^2)$ denotes the truncated normal distribution on the hyper-cube $C$. $\mathcal{U}(A)$ denotes uniform distribution on $A$.

%% file: 2-ProblemFormulation.tex
Suppose that an agent's behaviour is governed by a Markov control model $(X,A,\{A(x) \mid x\in X\},Q,\elltrue)$ \citep[Def.~2.2.1]{savorgnan2009discrete}, where
\begin{itemize}
\item $X\subset\mR^{n_x}$, $A\subset\mR^{n_a}$ are the Borel space for state and action, respectively;
\item $\{A(x) \mid x\in X\}\subset A$ is the set of feasible control. Further, let $\mK:=\{(x,a) \mid x\in X,a\in A(x)\}\subset X\times A$ be the set of feasible state-action pairs;
\item $Q:(\mathcal{B}(X)\times \mK)\rightarrow [0,1]$ is the transition kernel;
\item $\elltrue:\mK\rightarrow\mR$ is the running cost.
\end{itemize}
Further, let $(\Omega,\mathcal{F},\mP)$ be the probability space that carries the state-action sequence $\{(\mfx_t,\mfa_t)\}_{t=1}^\infty$.
We assume that the agent make its decision by seeking to optimize the following infinite-horizon expected total discounted cost
\begin{align}
\min_{\pi}J(\pi;\mu_1)=\mE\left[\sum_{t=1}^\infty\alpha^{t-1}\elltrue(\mfx_t,\mfa_t)\right],\label{eq:forward_problem_obj_fun}
\end{align}
where $\alpha\in(0,1)$ is the decay constant, and the initial value $x_1$ is distributed according to $\mu_1$ on $X$.
In particular, the optimal policy $\bm{\pi}=\{\pi_t\}$ is a sequence of stochastic kernels $\pi_t:A\times (\mK^{t-1}\times X)\rightarrow [0,1]$, and the action $a_t$ that the agent taken at each time step $t$ is sampled from the distribution $\pi_t(a|x_1,a_1,\cdots,x_{t-1},a_{t-1},x_t)$.

To proceed, we make the following assumptions.
\begin{assumption}[Markov transition kernel]\label{ass:transition_kernel}
 The Markov transition kernel is a priori known. In addition, the Markov transition kernel $Q(\cdot|x,a)$ is Lipschitz continuous in total-variation distance, i.e., there exists $L_Q\ge 0$ such that for all $\eta_1:=(x_1,a_1)\in\mathbb{K},\eta_2:=(x_2,a_2)\in\mathbb{K}$, 
\begin{align*}
\|Q(\cdot|\eta_1)-Q(\cdot|\eta_2)\|_{\tv}\le L_Q\|\eta_1-\eta_2\|_2.
\end{align*}
\end{assumption}
\begin{assumption}[State-action set hyper-cube]\label{ass:compact_support}
The state-action set $X\times A$ is a hyper cube, namely, $X=\times_{i=1}^{n_x}[x_{i,\min},x_{i,\max}]$ and $A = \times_{i=1}^{n_a}[a_{i,\min},a_{i,\max}]$. Moreover, for all $x\in X$, it hold for the feasible action set $A(x)$ that $A(x)=A$. Consequently, $\mK = X\times A$.
\end{assumption}
\begin{assumption}[Cost function]\label{ass:obj_func}
  The cost function $\elltrue$ has prior structures in the sense that it belongs to a function class that can be parameterized linearly by $\thetatrue_\ell\in\mR^{n_\ell}$, i.e., $\elltrue(\cdot,\cdot) = \thetatrue_\ell^T\varphi(\cdot,\cdot)$, where $\varphi(\cdot,\cdot) = [\varphi_1(\cdot,\cdot),\ldots,\varphi_{n_\ell}(\cdot,\cdot)]^T$ is the vector of (known) $L_\varphi$-Lipschitz continuous feature functions, $\elltrue \in \lip(L_\ell)$ where $L_\ell = \|\thetatrue\|_1 L_\varphi$.
\end{assumption}
Moreover, under the Assumptions \ref{ass:transition_kernel}, \ref{ass:compact_support}, \ref{ass:obj_func}, we can further reformulate the problem of optimizing \eqref{eq:forward_problem_obj_fun} as a linear optimization problem in the infinite-dimensional space of measures \citep[Chap.~6]{hernandez2012discrete}.
In particular, given a policy $\bm{\pi}$ and the initial value's distribution $\mu_1$, the state-control occupation measure at time step $t$ is defined as
\begin{align*}
  \mu_t^\pi(\Gamma_x\times \Gamma_a) :=\mE[\indicator_{\Gamma_x\times \Gamma_a}(\mfx_t,\mfa_t)],\quad \forall (\Gamma_x\times\Gamma_a)\subset\mK.
\end{align*}
In addition, let $\proj:\mathcal{M}_+(\mK)\rightarrow\mathcal{M}_+(X)$ be the linear projection operator that transform any probability measure $\mu\in\mathcal{M}_+(\mK)$ to its marginal distribution $\tilde{\mu}\in\mathcal{M}_+(X)$ on the state. 
Moreover, for all $\Gamma_x\subset X$, let $q:\mathcal{M}_+(\mK)\rightarrow\mathcal{M}_+(X)$ be the linear forward time shift operator defined by
\begin{align*}
  (q\mu_t^\pi)(\Gamma_x) = \int_{\mK} Q(\Gamma_x|x,a)\mu_t^\pi(\mathrm d x,\mathrm d a).
\end{align*}
Therefore, in view of the definition of the transition kernel $Q$ and the state-control occupation measure, it is clear that $\proj\mu_{t+1}^\pi = q\mu_t^\pi$. Let us further define the following occupation measure
\begin{align}\label{eq:mu^pi_def}
  \mu^\pi:=\sum_{t=1}^\infty \alpha^{t-1} \mu_t^\pi.
\end{align}
Equipped with the set-up above, we can reformulate the the problem of optimizing \eqref{eq:forward_problem_obj_fun} as
\begin{subequations}\label{eq:forward_problem_measure_form}
\begin{align}
  \min_{\mu^\pi\in\mathcal{M}_+(\mK)} &\;\;\langle \elltrue,\mu^\pi\rangle \label{eq:forward_problem_measure_form_obj}\\
  \st &\;\;(\proj-\alpha q)\mu^\pi = \mu_1.\label{eq:forward_problem_measure_form_dynamics}
\end{align}
\end{subequations}
In fact, the Assumptions \ref{ass:transition_kernel}, \ref{ass:compact_support}, \ref{ass:obj_func} guarantees that \eqref{eq:forward_problem_measure_form} is equivalent to optimizing \eqref{eq:forward_problem_obj_fun} \citep[Thm.~6.3.7]{hernandez2012discrete} in the sense that there is an equivalence between the feasible points of the two problems and that they both attains the same optimal value.
\footnote{More specifically, Assumption \ref{ass:transition_kernel}, \ref{ass:compact_support} and \ref{ass:obj_func} together imply \citep[Ass.~4.2.1]{hernandez2012discrete}.}
The equivalence between \eqref{eq:forward_problem_obj_fun} and \eqref{eq:forward_problem_measure_form} means that we can consider the latter formulation when we consider the IOC problem to come. 

To this end, the optimality conditions for the above problem are given in \eqref{eq:opt_cond_inf_dim}, and they can derived as follows. We first relax the constraint \eqref{eq:forward_problem_measure_form_dynamics} with a multiplier $\Vtrue \in \mathscr{F}(X)$, and get the Lagrangian $\mathcal{L} : \mathcal{M}_+(\mK) \times \mathscr{F}(X) \rightarrow \mR$
\begin{align*}
\mathcal{L}(\mu^\pi, \Vtrue) & = \langle \elltrue,\mu^\pi\rangle - \langle \Vtrue, (\proj-\alpha q)\mu^\pi -\mu_1 \rangle \\
& =  \langle \elltrue + (\alpha q^*-\proj^*)\Vtrue,\mu^\pi\rangle + \langle \Vtrue, \mu_1 \rangle.
\end{align*}
Considering the first term in the above equation, we note that $\inf_{\mu^\pi\in\mathcal{M}_+(\mK)}  \mathcal{L}(\mu^\pi, \Vtrue) > -\infty$ only if 
\begin{subequations}\label{eq:opt_cond_inf_dim}
\begin{equation}
\elltrue+(\alpha q^*-\proj^*)\Vtrue\ge 0. \label{eq:bellman_inequality}
\end{equation}
On the other hand, since the above equation holds, the infimum of the Lagrangian is attained for any $\mu^\pi$ such that
\begin{equation}\label{eq:complementary_slackness}
  \langle \elltrue + (\alpha q^*-\proj^*)\Vtrue,\mu^\pi\rangle = 0,
\end{equation}
for example $\mu^\pi = 0$. However, the latter is not primaly feasible, which gives us the last optimality condition, namely
\begin{align}
  \eqref{eq:forward_problem_measure_form_dynamics},\quad \mu^\pi\in\mathcal{M}_+(\mK).\label{eq:primal_feasible}
\end{align}
\end{subequations}
In these equations, the adjoint operators are defined by $q^*\Vtrue:\mK\mapsto\mR$ and $\proj^* \Vtrue:\mK\mapsto\mR$ via
\begin{subequations}  \label{eq:adjoints}
  \begin{align}
    (q^*\Vtrue)(x,a) &=\int_X \Vtrue(\chi)Q(d\chi|x,a)\label{eq:cost_to_go}\\
    (\proj^*\Vtrue)(\mfx,\mfu) &= \Vtrue(\mfx)\indicator_{\mK}(\mfx,\mfa)=\Vtrue(\mfx).\label{eq:value_func}
  \end{align}
\end{subequations}
We further assume that as an observer, we can observe $M$ trials of the agent.  More precisely, let the trajectories $(\mfx_t^i, \mfa_t^i)_{t=1}^{N}$  be Independently Identically Distributed (I.I.D) over $i=1:M$. The observed trials $\{(x_t^i,a_t^i)\}_{i=1}^M$ are the realizations of the I.I.D state-action pairs $\{(\mfx_t^i,\mfa_t^i)\}_{i=1}^M$.
Moreover, we make the following assumptions on the initial value distribution $\mu_1$ and the observed agent's optimal actions.

\begin{assumption}[Initial value distribution]\label{ass:init_value}
The support of the initial value distribution $\mu_1$ is $X$.
\end{assumption}

\begin{assumption}[Deterministic stationary policy]\label{ass:stationary_markov_policy}
The observed agent's optimal actions $\{a_t^i\}_{i=1}^M$ for all $t=1:N$ are generated from a deterministic stationary policy \citep[Def.~2.3.2]{hernandez2012discrete}. Namely, it holds for the optimal policy $\bm{\pitrue}=\{\pitrue_t\}$ that $\pitrue_t(a|x_1,a_1,\cdots,x_{t-1},a_{t-1},x_t) = \delta(\pitrue(x_t))$, which is the Dirac measure concentrated at the measurable function $\pitrue:X\rightarrow A$.
\end{assumption}

With the above set-up, we are ready to present the problem that is considered in this paper. 

\begin{problem}\label{pro:ioc}
  Given the $M$ observations of the optimal ``state-action pair" trajectories $\{(x_t^i,a_t^i)\}_{t=1}^N$ that are generated by \eqref{eq:forward_problem_obj_fun}, construct a consistent estimator of $\pitrue:X\rightarrow A$ that takes the structure of $\elltrue$ into account.
\end{problem}

Note that different cost functions may lead to identical optimal policies.
For instance, $\pi(x)=0$ is optimal to both $\ell_1(x,a)=a^2$ and $\ell_2(x,a)=a^4$.
Hence, $\thetatrue_\ell$ is in general not identifiable unless the feature functions vector $\varphi$ is carefully chosen.
Nevertheless, motivated by the practical need of ``learning from demonstration", in this paper we choose to not focus on the identifiability of IOC. Instead, the question of identifiability of the cost function is left for future research.
Despite this, we can still do consistency analysis in the sense that the collected state-action pairs $(x_t^i, a_t^i)$ are indeed optimal to the reconstructed cost function as the data amount tends to infinity. 
Moreover, we can also get an estimate $\hat{\pi}$ of the underlying optimal policy $\pitrue$ that corresponds to the estimated cost $\hat{\theta}_\ell^T\varphi$ (cf.~\citep{zhang2024inverse}).
Such policy estimation intrinsically takes the prior knowledge of the structures in $\elltrue$ into account, in contrast to ``behaviour cloning" \citep{sammut2017behavioral, codevilla2019exploring, torabi2018behavioralcloningobservation}. More precisely, behaviour cloning parameterizes a family of deterministic stationary policies $\pi(\cdot;\theta)$ (typically a neural network) and solves the following optimization problem
\begin{align}
\min_{\theta}\frac{1}{M}\sum_{i=1}^M\sum_{t=1}^N\|\mfa_t^i-\pi(\mfx_t^i;\theta)\|_2^2.
\label{eq:behaviour_cloning}
\end{align}
Indeed, it also offers a consistent estimator of $\pitrue$, yet such control policy reconstruction is vulnerable to out-of-distribution data samples.
Such samples can occur, e.g., due to compounding errors when using the control policy \citep{ross2011reduction}, due to distribution shifts in the noise \citep{pmlr-v70-pinto17a}, or when trying to generalize the behaviour to new environments \citep{pmlr-v48-finn16, fu2017learning}. 
In contrast, by taking the structures in $\elltrue$ into account, we obtain a more robust control policy estimate.

%% file: 4-IOCAlgorithm.tex
In this section, we construct the estimation algorithm for the cost function and the control policy. 
This is based on taking the nonlinear optimal control problem and lifing it into a infinite dimensional linear representation. 
More specifically, we first formulate an infinite dimensional linear programming problem whose optimal solution is a cost function that corresponds to the observed data's occupation measure.
Next, we formulate an equivalent infinite dimensional linear programming problem which eleviates the problem of infinite time-horizon observations.
We then approximate the problem of recovering the cost function so that it is numerically tractable, and finally, based on a reconstrction of the cost function, we formulate an estimator for the control policy. Consistency of the estimators are proved in Sec.~\ref{sec:consistency}.

\subsection{The infinite dimensional IOC problem}
First note that if the running cost function $\ell$ and its corresponding value function $V$ satisfy
\begin{align*}
  \ell(x,a)-V(x)+\alpha\int_X V(y)Q(\mathrm dy|x,a) = 0,
\end{align*}
then any state-action trajectory $\{(\mfx_t,\mfa_t)\}_{t=1}^\infty$ is optimal to $\ell$.
For example, when the running cost $\ell(x,a)=\xi$ and the value function $V(x) = \frac{\xi}{1-\alpha}$, for some $\xi\geq0$ and all $(x,a)\in\mK$.
To avoid this trivial and uninformative solution to the IOC problem, we further make the following assumption.
\begin{assumption}\label{ass:IOC_normalization}
The ``true" running cost and the ``true'' value function $(\elltrue,\Vtrue)$ satisfy
\[
\int_{\mK}\left[\elltrue(x,a)-\Vtrue(x)+\alpha\int_X \Vtrue(y)Q(\mathrm dy|x,a)\right]\mathrm dx \mathrm da\ge 1.
\]
\end{assumption}
The above assumption rules out the case that the running cost is constant since it requires the admissible state-action set $\mK$ to at least contain a non-zero measure set whose components are not optimal.

Next, we construct an IOC algorithm that is based on the optimality conditions \eqref{eq:opt_cond_inf_dim} and the adjoint operators \eqref{eq:adjoints}. To this end, we first assume that $\mu^{\pitrue}$ is known (we approximate it with data in Sec.~\ref{sec:finite-dim-ioc}). Now, for every $(x,a)\in\mK$, we define the violation of the complementary slackness condition \eqref{eq:complementary_slackness} as $\psi:\mK\rightarrow\mR$ where
\begin{equation}
  \psi(x,a;\theta_\ell,V)\!:=\!\theta_\ell^T\varphi(x,a)\! +\alpha (q^*V)(x,a)\!-\!V(x).
\label{eq:complementary_slackness_violation}
\end{equation}
Then, we construct the IOC algorithm by minimizing the ``violation" of \eqref{eq:complementary_slackness}. Namely,
\begin{subequations}\label{eq:IOC_origin}
\begin{align}
  \min_{\substack{\theta_\ell\in\mR^{n_\ell}, V\in\mathscr{F}(X)\\\psi\in\mathscr{F}(\mK)}} 
  &\;\langle\psi,\mu^{\pitrue} \rangle\label{eq:violation_of_complementarity_slackness}\\
  \st  \qquad &\;\psi = \!\theta_\ell^T\varphi\! +\alpha (q^*V)\!-\!\proj^*V\label{eq:psi_def}\\
  &\;\psi(x,a)\!\ge\! 0,\forall (x,a)\in\mK,\label{eq:bellman_inequality_constraint}\\
  &\int_{\mK}\psi(x,a)\mathrm{d}x\mathrm{d}a \ge 1.\label{eq:regularization}\\
  &\|\theta_\ell\|_{1}\!\le\! \beta_\ell,\: \|\psi\|_{\infty,\mK}\!\le\! \beta_\psi,\label{eq:l_psi_norm_bounds}\\
  &\|V\|_{\infty,X}\!\le\! \beta_V.\label{eq:V_psi_norm_bounds}
\end{align}
\end{subequations}
Note that the norm bounds $\beta_\ell$, $\beta_V$ and $\beta_\psi$ in \eqref{eq:l_psi_norm_bounds} and \eqref{eq:V_psi_norm_bounds} can be chosen to be arbitrarily large. The norm constraint guarantees that the solutions to the optimization problem is always bounded, i.e., that we have a bounded estimator, which is important in the analysis later. Nevertheless, even without the constraint \eqref{eq:l_psi_norm_bounds} and \eqref{eq:V_psi_norm_bounds}, the following Proposition still holds.

\begin{proposition}\label{prop:well_posed_ioc_infinite_dim}
The problem \eqref{eq:IOC_origin} attains at least one optimal solution, and the optimal value is $0$. Moreover, for any optimal solution $(\theta_\ell^\star, V^\star, \psi^\star)$, $\pitrue$ is an optimal control policy to the optimal control problem \eqref{eq:forward_problem_obj_fun} with the running cost $\ell^\star:=\theta_\ell^{\star T}\varphi$. 
\end{proposition}

\begin{proof}
First, $\mu^{\pitrue}\in\mathcal{M}_+(\mK)$ and the constraint \eqref{eq:bellman_inequality_constraint} ensure that the objective function \eqref{eq:violation_of_complementarity_slackness} is bounded below by zero. 
Moreover, the optimal value zero can be attained since the ``true" $(\elltrue=\thetatrue_\ell^T\varphi, \Vtrue, \elltrue +\alpha (q^*\Vtrue) - \proj^*\Vtrue)$ satisfies the optimality condition \eqref{eq:opt_cond_inf_dim} and thus is both feasible to \eqref{eq:IOC_origin} and have objective function value zero. 

Next, let $(\ell^\star, V^\star, \psi^\star)$ be any optimal solution to \eqref{eq:IOC_origin}. Then $0 = \langle\psi^\star,\mu^{\pitrue} \rangle = \langle \ell^\star +\alpha (q^*V^\star) - \proj^*V^\star, \mu^{\pitrue}\rangle$, and since constraint \eqref{eq:bellman_inequality_constraint} is also satsified, $\pitrue$ satisfies the optimality conditions \eqref{eq:opt_cond_inf_dim}.
\end{proof}

In addition, for bounded $\elltrue$, its corresponding value function $\Vtrue$ is the limit of the well-known Value Iteration (VI), which is a contraction map on $ \mathscr{F}(X) $ with respect to supremum norm \citep[p. 52]{hernandez2012discrete}. 
Thus we have $\left\|\Vtrue\right\|_\infty< \infty$. 
Furthermore, by Assumption \ref{ass:transition_kernel} and \ref{ass:obj_func}, $\Vtrue$ is the fixed point of VI, this implies that for all $\eta_1:=(x_1,a_1),\eta_2:=(x_2,a_2)\in\mathbb{K}$, it holds that
\begin{align*}
&\left|\min_{a_1\in A}(\elltrue+\alpha q^* \Vtrue)(x_1,a_1)-\min_{a_2\in A}(\elltrue+\alpha q^* \Vtrue)(x_2,a_2)\right|\\
&\le \sup_{a\in A}\left|(\elltrue+\alpha q^* \Vtrue)(x_1,a)-(\elltrue+\alpha q^* \Vtrue)(x_2,a)\right|\\
&\le\sup_{a\in A}\left[\left|\elltrue(x_1,a)-\elltrue(x_2,a)\right|+\alpha\left|q^*\Vtrue(x_1,a)-q^*\Vtrue(x_2,a)\right|\right]\\
&\le \sup_{a\in A}\left[(L_\ell+\|\Vtrue\|_\infty L_Q)\|[x_1^T,a^T]^T-[x_2^T,a^T]^T\|_2\right]\\
&=\underbrace{(L_\ell+\|\Vtrue\|_\infty L_Q)}_{\beta_C}\|x_1-x_2\|_2
\end{align*}

and hence $\Vtrue\in \text{Lip}(\beta_C)$. 
Thus we can restrict ourselves to the set $V\in \text{Lip}(\beta_C)$ and $\psi\in \contfun(\mK)$, and the ``true" $\Vtrue$ and the corresponding $\bar{\psi}$ are still feasible when solving \eqref{eq:IOC_origin}. That is, we instead consider the problem
\begin{subequations}\label{eq:IOC_origin_cont_func_set}
  \begin{align}
    \min_{\substack{\theta_\ell\in\mR^{n_\ell}, V\in \contfun(X)\\\psi\in \contfun(\mK)}} &\;\;\langle\psi,\mu^{\pitrue} \rangle\\
    \label{eq:IOC_origin_cont_func_set_obj}
    \st \qquad &\;\;\eqref{eq:psi_def}-\eqref{eq:V_psi_norm_bounds},\\
    &V\in \text{Lip}(\beta_C).
    \label{eq:IOC_origin_V_lip_constraint}
  \end{align}
\end{subequations}
Analogously to the constraints \eqref{eq:l_psi_norm_bounds} and \eqref{eq:V_psi_norm_bounds}, we do not have to know the exact number of $\beta_C$ in practice since such bound on Lipschitz constant can be taken arbitrarily large.

Nevertheless, we do not have access to the measure $\mu^{\pitrue}$ in practice. One reason is that $\mu^{\pitrue}$ is defined as a sum of infinitely many occupation measures $\mu_t^{\pitrue}$ over time, yet we do not have infinite data length in practice. However, we can show that finite-time observations of the agent are sufficient to achieve the same effect as $\mu^{\pitrue}$.

\begin{proposition}[Finite-time IOC algorithm]\label{prop:finite_time_horizon_equivalent}
  Under Assumption \ref{ass:stationary_markov_policy}, for any $\nu\in\mathcal{M}_+(X)$, if $\support(\nu) =  \support(\proj\mu^{\pitrue})$, then solving \eqref{eq:IOC_origin_cont_func_set} is equivalent to solving
  \begin{subequations}\label{eq:IOC_approx1}
  \begin{align}
    \min_{\substack{\theta_\ell\in\mR^{n_\ell}, V\in \contfun(X)\\\psi\in \contfun(\mK)}} &\;\;\langle\psi, \nu \otimes\pitrue  \rangle\label{eq:IOC_approx1_obj_func}\\
    \st \qquad &\;\;\eqref{eq:psi_def}-\eqref{eq:V_psi_norm_bounds},\eqref{eq:IOC_origin_V_lip_constraint}, \label{eq:IOC_approx1_constraints}
  \end{align}
\end{subequations}
  where 
  \begin{align*}
    \langle f, \nu\otimes\pitrue\rangle := \int_X \left[ \int_A f(x,a)\pitrue(\mathrm{d}a|x) \right]\nu(\mathrm{d}x),\quad \forall f.
  \end{align*}
More specially, it holds that $\hat{\mu}^{\pitrue,N}:=\sum_{t=1}^N\mu_t^{\pitrue} = \left(\sum_{t=1}^N\proj\mu_t^{\pitrue}\right)\otimes\pitrue$, and under Assumption \ref{ass:init_value}, it holds that $\support\left(\sum_{t=1}^N\proj\mu_t^{\pitrue}\right)=\support( \proj\mu^{\pitrue}).$
\end{proposition}

\begin{pf}
  Since $\support(\nu) =\support(\proj\mu^{\pitrue})$, for any $f:X\rightarrow \mR_+$, it holds that
  \begin{align*}
    \int_Xf(x)(\proj\mu^{\pitrue})(\mathrm{d}x) = 0
    \Leftrightarrow\int_Xf(x)\nu(\mathrm{d}x) = 0.
  \end{align*}
  By Propostion \ref{prop:well_posed_ioc_infinite_dim}, the optimal value of \eqref{eq:IOC_origin} is zero, and hence the optimal value of \eqref{eq:IOC_origin_cont_func_set} is also zero. By following the same argument as that in the proof of Proposition \ref{prop:well_posed_ioc_infinite_dim}, it follows from $\nu\otimes\pitrue\in\mathcal{M}_+(\mK)$ and the constraint \eqref{eq:bellman_inequality_constraint} that the optimal value \eqref{eq:IOC_approx1} is bounded from below by zero and that the ``true" $(\elltrue, \Vtrue, \elltrue +\alpha (q^*\Vtrue) - \proj^*\Vtrue)$ attains the optimal value zero. Now suppose $(\ell^\star,V^\star, \psi^\star)$ is optimal to \eqref{eq:IOC_origin_cont_func_set}.  Thus we have
  \begin{align*}
    &(\ell^\star,V^\star, \psi^\star) \text{ minimizes \eqref{eq:IOC_origin},}\\
    \Rightarrow&\langle\psi^\star, \mu^{\pitrue}\rangle = \int_X \left[ \int_A \psi^\star(x,a)\pitrue(\mathrm{d}a|x) \right](\proj\mu^{\pitrue})(\mathrm{d}x) = 0, \\
    \Leftrightarrow&\langle\psi^\star,\nu\otimes\pitrue\rangle = \int_X \left[ \int_A \psi^\star(x,a)\pitrue(\mathrm{d}a|x) \right]\nu(\mathrm{d}x) = 0.
  \end{align*}
  Since \eqref{eq:IOC_origin_cont_func_set} and \eqref{eq:IOC_approx1} have the same constraints, $(\ell^\star,V^\star, \psi^\star)$ is also optimal to \eqref{eq:IOC_approx1}. Conversely, the same argument applies.
\end{pf}

\subsection{The finite dimension approximation of the IOC algorithm}\label{sec:finite-dim-ioc}
The above proposition implies that, instead of solving \eqref{eq:IOC_origin_cont_func_set}, we can spare ourselves from infinite data length and solve \eqref{eq:IOC_approx1} with $\nu\otimes\pitrue=\hat{\mu}^{\pitrue,N}$. Yet we still do not know the distribution $\hat{\mu}^{\pitrue,N}$ in practice, and hence we need to further approximate \eqref{eq:IOC_approx1} with $M$ I.I.D. samples $\{(x_t^i,a_t^i)\}_{t=1}^N$, $i=1:M$. In particular, we approximate the objective function \eqref{eq:IOC_approx1_obj_func} with empirical average, i.e.,
\begin{align}
  \langle\psi,\hat{\mu}^{\pitrue,N} \rangle\approx\frac{1}{M}\sum_{i=1}^M\sum_{t=1}^N\psi(\mfx_t^i,\mfa_t^i).
  \label{eq:emphirical_average}
\end{align}
Nevertheless, the approximated problem still optimizes over $V\in\contfun(X)$ and $\psi\in\contfun(\mathbb{K})$, which is infinite dimensional and thus difficult to solve. Moreover, the inequality constraint \eqref{eq:bellman_inequality_constraint} is challenging to handle since it needs to hold for all $(x,a)\in\mK$.
To address this issue, we approximate $\psi$ with a polynomial and use Putinar's Positivstellensatz \citep{putinar1993positive}\citep[Thm.~3.20]{laurent2009sums} to make the polynomial nonnegative.
More precisely, we approximate $\psi$ by a $2d_\psi$-th order polynomial on $\mK$ (without loss of generality, we consider only even-order polynomials to simplify the analysis to come), and represent the polynomial with Lagrange interpolation basis \citep{sauer1995multivariate}. Namely, for all $\eta := (x,a)\in\mK$, we have
\begin{align}\label{eq:polynomial_approx}
    \psi(\eta)\approx \hat{\psi}(\eta):= \underbrace{
    \begin{bmatrix}
    \hat\psi(\eta_{[1]}) &\cdots &\hat\psi(\eta_{[D_\psi]})
    \end{bmatrix}
    }_{\theta_\psi^T}\phi(\eta;\eta_{[\cdot]}),
\end{align}
where $D_\psi = C(n_x+n_u+2d_\psi, 2d_\psi)$ is the dimension of $2d_{\psi}$-th order polynomial space, $\eta_{[\cdot]}=\{\eta_{[1]},\ldots,\eta_{[D_\psi]}\}$ are the interpolation nodes, and the Lagrange polynomial basis $\phi(\eta;\eta_{[\cdot]})$ takes the form
\[
\phi(\eta;\eta_{[\cdot]}) = 
\begin{bmatrix}
\phi_1(\eta;\eta_{[\cdot]}) &\cdots &\phi_{D_\psi}(\eta;\eta_{[\cdot]})
\end{bmatrix}^T.
\]

\begin{remark}\label{remark:using-interpolation}
Note that any polynomial basis would work. Nevertheless, here we choose to work with Lagrange interpolation basis $\phi(\eta;\eta_{[\cdot]})$. The reason is that this leads to computation benefits; more details about this is found in Sec. \ref{sec:practical_implementation}.
\end{remark}

Next, to enforce the nonnegativity of $\hat{\psi}$, note that by Assumption \ref{ass:compact_support}, $\mK$ can be characterized by the semi-algebraic set $\{(x,a)|g_i(x,a)\ge 0,i=1:n_x+n_a\}$, where $g_i(x,a)= (x_i-x_{i,\min})(x_{i,\max}-x_i)$ and analogously for the bounds on $a$.
The quadratic module \citep[Eq.~3.13]{laurent2009sums} generated by $g_1,\ldots,g_{n_x+n_a}$ takes the form
\begin{align*}
\mathscr{Q}(g_{1:n_x+n_a})\!:=\!\{u_0\!+\!\!\!\sum_{i=1}^{n_x+n_a}\!\!u_i g_i \mid u_0,u_i\!\in\!\Sigma^2\}.
\end{align*}
Notably, $\hat{\psi}\in \mathscr{Q}(g_{1:n_x+n_a})$ means that $\hat{\psi}(\eta)\ge 0,\forall \eta\in\mK$.
Nevertheless, constraining $\hat{\psi}$ to be in $\mathscr{Q}(g_{1:n_x+n_a})$ might, a prior, be too restrictive. To see that it is in fact not the case, we first observe that $\mK$ is also Archimedean \citep[Def.~3.18]{laurent2009sums}. To see this, we consider the function 
\begin{align*}
f(\eta) = \sum_{i=1}^{n_x+n_a}g_i(x,a)=-\|\eta-\zeta\|_2^2+\mbox{constant},
\end{align*}
where $\zeta = [\zeta_1^x,\ldots,\zeta_{n_x}^x,\zeta_{1}^a,\ldots,\zeta_{n_a}^a]^T$ and $\zeta_i^x = \frac{x_{i,\min}+x_{i,\max}}{2}$,  $\zeta_i^a = \frac{a_{i,\min}+a_{i,\max}}{2}$. It holds that $f\in\mathscr{Q}(h_{1:n_x+n_a},g_{1:m})$. Moreover, the set $\{\eta|f(\eta)\ge 0\}$ is compact. Hence by \citep[Thm.~3.17, Def.~3.18]{laurent2009sums}, $\mK$ is Archimedean. 
By Putinar's Positivstellensatz \citep{putinar1993positive}\citep[Thm.~3.20]{laurent2009sums}, all polynomials that are positive on $\mK$ must belong to $\mathscr{Q}(g_{1:n_x+n_a})$.
In particular, this means that 
\begin{align}
\hat{\psi}(\eta)> 0, \; \forall \eta\in \mK \quad
\implies \quad \hat{\psi}\in\mathscr{Q}(g_{1:n_x+n_a}).\label{eq:epsilon_ineq1}
\end{align}
Therefore we can approximate \eqref{eq:bellman_inequality_constraint} to abitrary accuracy with \eqref{eq:epsilon_ineq1}. 
On the other hand, since \eqref{eq:epsilon_ineq1} says that $\hat{\psi}$ should belong to the quadratic module, it implies that the coefficients $\theta_\psi$ must belong to the cone $\left\{\theta\in\mR^{D_\psi}|\theta^T\phi\in \mathscr{Q}(g_{1:n_x+n_a})\right\}$, which, with a slight abuse of notation, we will denote as $\theta_\psi \in \mathscr{Q}(g_{1:n_x+n_a})$.

With the polynomial approximation \eqref{eq:polynomial_approx}, we can re-write the empirical average \eqref{eq:emphirical_average} as
\begin{align*}
    &\frac{1}{M}\sum_{i=1}^M\sum_{t=1}^N\hat\psi(\mfx_t^i,\mfa_t^i) = h^T\theta_\psi,
\end{align*}
where 
\begin{align*}
\bm h = \begin{bmatrix}
        \frac{1}{M}\sum_{i=1}^M\sum_{t=1}^N\phi(\mfx_t^i,\mfa_t^i;\eta_{[\cdot]})
    \end{bmatrix}.
\end{align*}
Similarly, we can further approximate the constraint \eqref{eq:regularization} as
\begin{align*}
    \int_\mK \hat\psi(\eta) \mathrm{d}\eta = \theta_\psi^T\underbrace{\begin{bmatrix}
        \int_\mK\phi(\eta;\eta_{[\cdot]})\mathrm{d}\eta
    \end{bmatrix}}_{d} \ge 1.
\end{align*}
In addition, the structure of the constraint \eqref{eq:psi_def} lends itself to an approximation by linear equations.
To see this, noticing that $V$ only depends on $x$, we hence approximate $V$ by $2d_V$-th order polynomial on $X$,
\begin{align*}
    V(x) \approx \hat{V}(x):=\theta_V^T r(x), \; \theta_V \in \mR^{D_V},
\end{align*}
where $D_V = C(n_x+2d_V, 2d_V)$, and $r(x)$ is a basis of $2d_V$-th order polynomials on $\mR^{n_x}$.
Consequently, the coordinate of $\proj^*V$ and $q^*V$ with respect to basis $\phi(\eta;\eta_{[\cdot]})$ can be calculated by their value at the interpolation points, i.e.,
\begin{align*}
    (\proj^*V)(\eta) &= \theta_V^T\underbrace{\begin{bmatrix}
        (\proj^*r)(\eta_{[1]}) 
        &\cdots 
        &(\proj^*r)(\eta_{[D_\psi]})
    \end{bmatrix}}_{=: G_1^T}\phi(\eta;\eta_{[\cdot]}),\\
    (q^*V)(\eta) &= \theta_V^T\underbrace{\begin{bmatrix}
        (q^*r)(\eta_{[1]}) 
        &\cdots 
        &(q^*r)(\eta_{[D_\psi]})
    \end{bmatrix}}_{=: G_2^T}\phi(\eta;\eta_{[\cdot]}).
\end{align*}
In a similar manner, we further do a coordinate change on $\ell$ to the same basis $\phi(\eta;\eta_{[\cdot]})$.
More specifically, by Assumption \ref{ass:obj_func}, 
\begin{align*}
\ell = \theta_\ell^T\varphi =\theta_\ell^T
\begin{bmatrix}
\varphi_1\\\vdots\\\varphi_{n_\ell}
\end{bmatrix}
\approx \theta_\ell^T
\underbrace{\begin{bmatrix}
\gamma_{\varphi_1}^T\\\vdots\\\gamma_{\varphi_{n_\ell}}^T
\end{bmatrix}}_{=: H^T}\phi,
\end{align*}
where $\gamma_{\varphi_i}^T\phi(\eta;\eta_{[\cdot]})$ is an approximation of the feature $\varphi_i$ in the Lagrange polynomial basis $\phi(\eta;\eta_{[\cdot]})$.
Therefore, the constraint \eqref{eq:psi_def} now can be approximated with the following linear constraints on the coefficients of the polynomial approximations
\begin{align}
    &\theta_\psi\!=\!\! \underbrace{\begin{bmatrix}H &\alpha G_2\!-\!G_1\end{bmatrix}}_{\Xi_\psi \in \mR^{D_\psi\times (n_\ell + D_V)}}\! \begin{bmatrix}
        \theta_\ell\\\theta_V
    \end{bmatrix}\!\!\Rightarrow\!\!
    \underbrace{\begin{bmatrix}
    I &-H &G_1\!\!-\!\!\alpha G_2
    \end{bmatrix}}_{=: G}\!
    \begin{bmatrix}
    \theta_\psi \\\theta_\ell \\ \theta_V
    \end{bmatrix}\!=\!0.\label{eq:psi_ell_V_equality_constraint}
\end{align}

As a brief summary, we can approximate \eqref{eq:IOC_origin} with the convex optimization problem
\begin{subequations}\label{eq:IOC_approx2}
\begin{align}
\min_{\theta_\psi\theta_\ell,\theta_V}&\;\;\bm h^T\theta_\psi\\
&\;\; \eqref{eq:psi_ell_V_equality_constraint}\\
&\;\; \theta_\psi\in\mathscr{Q}(g_{1:n_x+n_a}),\label{eq:theta_psi_constraint}\\
&\;\; d^T\theta_\psi\geq 1,\\
&\;\; \|\theta_\psi\|_{\infty}^2\le \beta_\psi^\prime,\label{eq:theta_psi_bounded}\\
&\;\;\|\theta_\ell\|_{\infty}^2\le \beta_\ell^\prime,\;\|\theta_V\|_{\infty}^2\le \beta_V^\prime.\label{eq:theta_V_ell_bounded}
\end{align}
\end{subequations}
Moreover, since all norms on finite dimensional vector spaces are equivalent \citep[Theorem~2.4-5]{kreyszig1991introductory}, the constraints \eqref{eq:l_psi_norm_bounds},\eqref{eq:V_psi_norm_bounds} and \eqref{eq:IOC_origin_V_lip_constraint} approximated in finite dimension is relaxed and simplized to \eqref{eq:theta_psi_bounded} and \eqref{eq:theta_V_ell_bounded}, where $\beta_\psi^\prime,\beta_\ell^\prime,\beta_V^\prime$ can be chosen arbitrarily large in practice.

\subsection{Expert policy reconstruction}\label{sec:policy-recovery}
Solving \eqref{eq:IOC_approx2} gives the esimates $(\hat{\ell}_{M,\mathbf{d}},\hat{V}_{M,\mathbf{d}})$ for the ``true" running cost and the value function $(\elltrue,\Vtrue)$, where $\mathbf{d}=(d_\psi,d_V)$. 
With the identified $\hat{\ell}_{M,\mathbf{d}}$, any off-the-shelf algorithm (e.g., model predictive control) that solves the forward optimal control problem can be used for control policy reconstruction.
As mentioned previously in Sec.~\ref{sec:problem_formulation}, such control policy reconstruction method should have better generalizability and robustness compared to behaviour cloning.
Nevertheless, in the worst case scenarios,
the reconstructed control policy may be different from the expert's policy.
This is because the ``true" cost function may not be identifiable and the identified cost function does not align with the ``true" one; or because the ``true" cost function may admit multiple optimal control policies.
Despite this, we can still get a reasonable control policy by taking the prior knowledge of the cost structure and the IOC result into account in the worst case scenarios.
To this end, note that by Assumption \ref{ass:stationary_markov_policy} and \eqref{eq:opt_cond_inf_dim}, it holds for the ``true" $(\elltrue,\Vtrue)$ and the expert control policy $\pitrue$ that
\begin{align*}
0&=\langle \bar{\psi}:=\elltrue+(\alpha q^*-\proj^*)\Vtrue,\mu^{\pitrue}\rangle+ \langle \| a - \pitrue(x) \|_2^2],\mu^{\pitrue}\rangle\\
&=\langle \bar{\psi},(\proj \mu^{\pitrue})\otimes\pitrue\rangle+\langle \| a - \pitrue(x) \|_2^2],\mu^{\pitrue}\rangle\\
&=\int_X \bar{\psi}(x, \pitrue(x)) (\proj\mu^{\pitrue})(\mathrm{d}x) + \langle\| a - \pitrue(x) \|_2^2,\mu^{\pitrue}\rangle.
\end{align*}
On the other hand, due to \eqref{eq:bellman_inequality_constraint}, for expert policy $\pitrue(x)$ and all deterministic stationary policies $\pi$ it holds that
\[
\int_X \bar{\psi}(x, \pi(x)) (\proj\mu^{\pitrue})(\mathrm{d}x) +\langle\| \pitrue(x) - \pi(x) \|_2^2,\mu^{\pitrue}\rangle\ge 0.
\]
Hence we parameterize $\pi$ by $\theta_\pi\in\Theta$, where $\Theta\subset\mR$ is compact and $\pi(x;\cdot):\Theta\rightarrow A$ is continuous for any $x$, for instance, a neural network $\pi(x;\theta_\pi)$ whose weights are $\theta_\pi$. And then we estimate $\pitrue$ by solving
\[
\min_{\theta_\pi\in\Theta} \; \langle\bar{\psi}(x, \pi(x;\theta_\pi)) +\| \pitrue(x) - \pi(x;\theta_\pi) \|_2^2,\proj\mu^{\pitrue}\rangle.
\]
Notably, by following an argument similar to the proof of Proposition \ref{prop:finite_time_horizon_equivalent}, it is equivalent to solving
\begin{align}
\min_{\theta_\pi} \; & \langle\bar{\psi}(x, \pi(x;\theta_\pi))+\| \pitrue(x) - \pi(x;\theta_\pi) \|_2^2,\proj\hat{\mu}_t^{\pitrue,N}\rangle
\label{eq:policy-approx-reconstruction-infinite}
\end{align}
Though we do not have access to neither the actual distribution $\hat{\mu}_t^{\pitrue,N}$, nor $\bar{\psi}$, we can still approximate the above problem. First, the integral can be approximated with the empirical mean of the I.I.D. observations of the optimal ``state-action pair" trajectories $\{(x_t^i,a_t^i)\}_{t=1}^N$, $i=1:M$. Next, we approximate $\bar{\psi}$ with the approximation $\hat{\psi}_{M,\mathbf{d}}$ that solves \eqref{eq:IOC_approx2}. Together, this gives the finite-dimensional optimization problem
\begin{align}
\min_{\theta_\pi\in \Theta} \! \frac{1}{M}\!\sum_{t=1}^N\!\sum_{i=1}^M\! \hat{\psi}_{M,\mathbf{d}}(\mfx_t^i, \pi(\mfx_t^i;\theta_\pi)) \!+\! \| \mfa_t^i \!-\! \pi(\mfx_t^i;\theta_\pi) \|_2^2.
\label{eq:policy_approx_reconstruction}
\end{align}

As a summary, we have constructed the infinite dimensional IOC algorithm \eqref{eq:IOC_approx1} and

a finite dimensional approximation \eqref{eq:IOC_approx2} that

can be solved numerically. 
More specifically, solving \eqref{eq:IOC_approx2} yields an approximation of the cost function as well as the value function that can lead to the ``true" policy.
As mentioned previously, with the identified cost function, one may apply any standard optimal control method to obtain a policy with good generalization capability.
In addition, solving \eqref{eq:policy_approx_reconstruction} provides another way of estimating the optimal control policy. The complete estimation procedure is described in Algorithm~\ref{alg:estimator}. 

\begin{algorithm}[!htpb]
    \caption{Estimation procedure for cost and policy}
    \label{alg:estimator}
    \begin{algorithmic}[1]
        \Require Data $\{(x_t^i,a_t^i)\}_{t=1}^N$, $i=1:M$, system dynamics $Q(\cdot|x,a)$, state and action space $\mK = X\times A$, polynomial degrees $(d_\psi,d_V)$, parametrization $\pi(\cdot;\theta_\pi)$.
        \Ensure $(\hat{\ell}_{M,\mathbf{d}},\hat{V}_{M,\mathbf{d}})$ and $\pi(\cdot;\hat{\theta}_\pi)$.
        \State Design interpolation the points $\eta_{[\cdot]}$ and the polynomial bases $r(x)$.
        \State Calculate the matrix $G$, defined in \eqref{eq:psi_ell_V_equality_constraint}, from system dynamics and sample moments vector $h$ from data.
        \State Obtain the estimate $(\hat{\ell}_{M,\mathbf{d}},\hat{V}_{M,\mathbf{d}})$ by solving \eqref{eq:IOC_approx2}.
        \State Obtain the estimate $\pi(\cdot;\hat{\theta}_\pi)$: either by solving \eqref{eq:policy_approx_reconstruction} to estimate the expert policy, or by applying any standard optimal control algorithm.
    \end{algorithmic}
\end{algorithm}

%% file: 5-ConsistenceAnalysis.tex
We now analyze the consistency of the proposed algorithm. 
By Proposition \ref{prop:finite_time_horizon_equivalent}, in theory, the underlying cost function can be estimated by solving \eqref{eq:IOC_approx1}. Yet such infinite dimensional optimization problem is not numerically tractable and hence the decision variables $\psi$ and $V$ in \eqref{eq:IOC_approx1}

are approximated by finite-order polynomials so as to yield problem \eqref{eq:IOC_approx2}. In addition, the integrals are approximated via the empirical averages defined in \eqref{eq:emphirical_average} and \eqref{eq:policy_approx_reconstruction}. To justify these approximations, we present the following theorem, which states that as the data amount $M$ and the polynomial orders $\mathbf{d}=(d_\psi, d_V)$ tend to infinity, the estimate $(\hat{\ell}_{M,\mathbf{d}}, \hat{V}_{M,\mathbf{d}})$ converges to a running cost and a value function that corresponds to the ``true'' expert policy $\pitrue$. Such convergence constitutes the key consistency property.

\begin{remark}
Note that in \eqref{eq:IOC_origin_cont_func_set}, we optimize $(\theta_\ell, V, \psi)$.
To simplify the analysis, in view of \eqref{eq:psi_def}, we substitute $\psi$ with the expression of $\theta_\ell$ and $V$ in the remainder of this section. 
\end{remark}

\begin{theorem}\label{thm: MAIN-CONSISTENCY-THEOREM}
    Under Assumptions~\ref{ass:transition_kernel}-\ref{ass:stationary_markov_policy} and \ref{ass:IOC_normalization}, as the polynomial degrees $d_\psi, d_V$ and the data sample size $M$ approach infinity, the solution of \eqref{eq:IOC_approx2} achieves the optimal value of \eqref{eq:IOC_approx1} almost surely. Formally, the following equation holds a.s. 
    \begin{equation*}
        \begin{aligned}
            \lim_{M\rightarrow\infty}\lim_{d_V\rightarrow\infty}\lim_{d_\psi\rightarrow\infty} \langle\theta_{\ell, M}^{\mathbf{d},T}\varphi\!+\!(\alpha q^*\!-\!\proj^*)(\theta_{V, M}^{\mathbf{d},T}r),\hat{\mu}^{\pitrue, N} \rangle\!=\!0,
        \end{aligned}
    \end{equation*}
    where $(\theta_{\ell, M}^{\mathbf{d}}, \theta_{V, M}^{\mathbf{d}})$ is an optimal solution of \eqref{eq:IOC_approx2}. 
\end{theorem}

\begin{remark}
The above theorem states the fact that, 
the occupation measure $\hat{\mu}^{\pitrue, N}$ that corresponds to the observed optimal state-action pairs $\{(x_t^i,a_t^i)\}_{t=1}^N$, $i=1,\ldots,M$ is ``nearly" optimal to the approximated cost function $\theta_{\ell, M}^{\mathbf{d},T}\;\varphi$ and the approximated value function $\theta_{V, M}^{\mathbf{d},T}\;r$. This is because the optimality conditions \eqref{eq:opt_cond_inf_dim} are ``nearly" satisfied.
\end{remark}

\begin{remark}
The convergence order in the above theorem: ``first increase the polynomial orders, then increase the data sample volume", actually coincide with the practical engineering practice. In particular, since data collection is typically expensive, if we get a bad estimate of the cost function, we would first increase the polynomial order to sufficiently high before we decide to collect more data.

\end{remark}

To prove Theorem \ref{thm: MAIN-CONSISTENCY-THEOREM}, we decompose the overall analysis into Lemma \ref{lem: stochastic-consistency-for-infinite-dim}, \ref{lem: Approximating V by polynomials} and \ref{lem: Approximating phi by polynomials}. 
First, the following lemma shows that the ``uniform law of large numbers" still holds in this infinite dimension function space.

\begin{lemma}\label{lem: stochastic-consistency-for-infinite-dim}
    Let $\left\{\eta_i\right\}_{i=1}^M$ be independent samples drawn from probability distribution $\nu$. Consider the stochastic optimization problem $\min_{\substack{f\in\mathcal{D}}} \mathbb{E}_{\bm{\eta} \sim \nu}[f(\bm{\eta})]$ and its corresponding empirical approximation problem $\min_{\substack{f\in\mathcal{D}}} \frac{1}{M}\sum_{i=1}^{M} f(\bm\eta_i)$, where $\mathcal{D}\subset \mathcal C(\mathbb{K})$ is equicontinuous \citep[Def~11.27]{rudin1987real}, pointwise bounded and closed.
    Then it holds that 
    \begin{equation*}
        \lim_{M\rightarrow\infty}\sup_{f\in\mathcal{D}}\left|\frac{1}{M}\sum_{i=1}^{M} f(\bm\eta_i)-\mathbb{E}[f]\right|=0, \text{a.s.}
    \end{equation*}
    Furthermore, let $f^\star$ and $f^\star_M$ denote the optimal solution of the original problem and the approximation problem, respectively. Then it holds that $\mathbb{E}_{\bm{\eta}\sim \nu}[f^\star_M(\bm{\eta})] \xrightarrow{a.s.}\mathbb{E}_{\bm{\eta}\sim \nu}[f^\star(\bm{\eta})]$. 
\end{lemma}
\begin{pf}
    By Arzela-Ascoli Theorem \citep[Thm~11.28]{rudin1987real}, every sequence in $\mathcal{D}$ has a uniform convergent subsequence on $\mK$, and hence $\mathcal{D}$ is compact. Thus $\mathcal{D}$ admits a finite cover number \citep[Def~2.1.5]{wellner2013weak} $N(\epsilon, \mathcal{D}, \|\cdot\|_\infty)$ for any $\epsilon > 0$. 
    Next, pick the centers $\left\{f_j\right\}_{j=1}^N$, it holds that $\forall f\in \mathcal{D}$, there is $f_j$, such that $\|f-f_j\|_\infty\leq\epsilon$. Let $E_M[f]:=\frac{1}{M}\sum_{i=1}^{M} f(\bm\eta_i)$, it holds that
    \begin{align*}
        &|E_M[f]\!-\!\mathbb{E}[f]|\\
        &=|E_M[f]\!-\!\mE[f_j]\!+\!\mE[f_j]\!-\!E_M[f_j]\!+\!E_M[f_j]\!-\!\mE[f_j]|\\
        &\le \left|E_M[f_j]\!-\!\mathbb{E}[f_j]\right|\!+\!|\mE[f_j]\!-\!\mE[f]]|\!+\!|E_M[f]\!-\!E_M[f_j]|\\
        &\le \left|E_M[f_j]\!-\!\mathbb{E}[f_j]\right|\!+\!\!\int\underbrace{|f_j-f|}_{\le \epsilon}d\nu\!+\!\frac{1}{M}\sum_{i=1}^M\underbrace{|f(\bm\eta_i)-f_j(\bm\eta_i)|}_{\le \epsilon}\\
        &\leq \left|E_M[f_j]-\mathbb{E}[f_j]\right|+2\epsilon.
    \end{align*}
    Therefore, it holds that $\sup_{f\in\mathcal D}|E_M[f]-\mE[f]|\le \max_{j=1:N}|E_M[f_j]\!-\!\mE[f_j]|+2\epsilon$.
    By strong laws of large numbers \citep[Thm~5.23]{kallenberg1997foundations}, as $M\rightarrow\infty$, it holds that $|E_M[f_j]-\mE[f_j]|\rightarrow 0$ a.s. for all $j=1:N$. And hence $\max_{j=1:N} |E_M[f_j]-\mE[f_j]|\rightarrow 0$ a.s..
    Thus $\limsup_{M\rightarrow\infty}\sup_{f\in\mathcal{D}}\left|E_M[f]-\mathbb{E}[f]\right|\leq 2\epsilon, \text{a.s.}$. Let $\epsilon\downarrow0$, then we have $\sup_{f\in\mathcal{D}}\left|E_M[f]-\mathbb{E}[f]\right|\xrightarrow{a.s.} 0$ as $M\rightarrow \infty$.

    Furthermore, note that the original and the approximated problems share the same feasible domain $\mathcal D$. And consequently, by the optimality of $f^\star$ and $f^\star_M$, respectively, it holds that $\mE[f^\star] \leq \mE[f^\star_M],E_M[f^\star_M] \leq E_M[f^\star]$. Therefore, it follows that
    \begin{align*}
        &\mathbb{E}[f^\star] \leq \mathbb{E}[f^\star_M] \leq E_M[f^\star_M] + \left|E_M[f^\star_M]-\mathbb{E}[f^\star_M]\right|\\
        &\leq E_M[f^\star] + \left|E_M[f^\star_M]-\mathbb{E}[f^\star_M]\right|\\
        &\leq \mathbb{E}[f^\star] + \underbrace{\left|E_M[f^\star]-\mathbb{E}[f^\star]\right|}_{\xrightarrow{a.s.}0}+\underbrace{\left|E_M[f^\star_M]-\mathbb{E}[f^\star_M]\right|}_{\xrightarrow{a.s.}0}.
    \end{align*}
    Thus we have $\mathbb{E}_{\bm{\eta}\sim \mu}[f^\star_M(\bm{\eta})] \xrightarrow{a.s.}\mathbb{E}_{\bm{\eta}\sim \mu}[f^\star(\bm{\eta})]$.
\end{pf}

The following Lemma \ref{lem: Approximating V by polynomials} shows that we can approximate $V$ with polynomials, reducing the problem to finite dimensions.

\begin{lemma}\label{lem: Approximating V by polynomials}
    Consider the optimization problem
    \begin{subequations}\label{eq:consistencyOriginal}
    \begin{align}
    \min_{\substack{\theta_\ell\in\mR^{n_\ell}\\ V\in \mathcal C(X)}} 
    &\;\langle\theta_\ell^T\varphi + T(V),\mu \rangle\\
    \st \quad &\;\theta_\ell^T\varphi + T(V)\!\ge\! 0,\forall (x,a)\in\mK,\label{eq:consistencyOriginal:cstr1}\\
    &\;\left\|\theta_\ell^T\varphi + T(V)\right\|_1 \geq 1,\label{eq:consistencyOriginal:cstr2}\\
    &\;\left\|\theta_\ell\right\|_1\leq \beta_\ell',\\
    &\;\left\|V\right\|_\infty\leq \beta_V',\label{eq:consistencyOriginal:cstr4}\\
    &\;V\in \lip(\beta_C)\label{eq:consistencyOriginal:cstr5},
    \end{align}
    \end{subequations}
    where $\mu\in \mathcal{M}_+$, $T:\mathcal C(X)\rightarrow \mathcal C(\mathbb{K})$ is a bounded linear operator.
On the other hand, its approximated problem takes the form
\begin{subequations}\label{eq:consistencyApproxStep2}
\begin{align}
    \min_{\substack{\theta_\ell\in\mR^{n_\ell}\\ V\in P_d(X)}} 
    &\;\langle\theta_\ell^T\varphi + T(V),\mu \rangle\\
    \st \quad &\;\eqref{eq:consistencyOriginal:cstr1}-\eqref{eq:consistencyOriginal:cstr5},\label{eq:consistencyApprox_cstr}
\end{align}
\end{subequations}
where $\left\{P_d\subset \mathcal C(X)\right\}_d^\infty$ is the set of $d$-order polynomials.
Let $(\theta_\ell^\star, V^\star)$ and $(\theta_{\ell, d}^\star , V_d^\star)$ be the optimal solution of \eqref{eq:consistencyOriginal} and \eqref{eq:consistencyApproxStep2}, respectively. If the feasible set of \eqref{eq:consistencyOriginal} has a strictly feasible solution, then it holds that
\begin{equation*}
    \lim_{d\rightarrow\infty}\langle\theta_{\ell, d}^{\star T}\varphi \!+\! T(V_d^\star),\mu \rangle = \langle\theta_\ell^{\star T}{\varphi} \!+\! T(V^\star),\mu \rangle.
\end{equation*}
\end{lemma}
\begin{pf}
    
    Since \eqref{eq:consistencyOriginal} is convex and its feasible set has a non-empty interior, for all $\epsilon$, there must exist a strictly feasible solution $(\theta_0,V_0)$ and $\epsilon_1$ such that $\forall (\theta,V)\in \{(\theta,V)\in\mR^{n_\ell}\times \mathcal{C}(X)\;|\; \|\theta - \theta_0\|_2+\|V-V_0\|_{\infty,\mK}< \epsilon_1\}$ is still feasible
    and $\langle (\theta_0 - \theta_\ell^\star)^T\varphi + T(V_0-V^\star), \mu \rangle \leq \epsilon/2$.
    Moreover, it also holds that $V_0\in\lip(\beta_C-\epsilon_C)$ for some $\epsilon_C>0$.
    Next, we first extend $V_0$ from domain $\mK$ to $\mR^n$ by Mcshane's theorem \citep{mcshane1934extension}, namely,
    \begin{align*}
    \tilde{V}_0(x) = \min_{x\in\mK} \{V_0(y)+(\beta_C-\epsilon_C)\|x-y\|\},
    \end{align*}
    and notably, $\tilde{V}_0=V_0$, $\forall x\in\mK$, and $\tilde{V}_0\in\lip(\beta_C-\epsilon_C)$.
    Next, let $\rho_\delta(x)$ be a mollifier and we mollify $\tilde{V}_0$ by $V_\delta(x) = (\tilde V_0\circledast \rho_\delta)(x)$, where $\circledast$ denotes the convolution.
    Hence 
    there exists a $\delta_\epsilon$, such that $\|V_{\delta_\epsilon}-\tilde V_0\|_{\infty,\mK}\le\min\{\epsilon/(4\|T^*\|\|\mu\|_{\tv}),\epsilon_1/2\}$, where $\|T^*\|$ is the operator norm of the dual $T^*$.
    
    Furthermore, it also preserves the Lipschitz constant, namely,
    \begin{align*}
    &\left|V_{\delta_\epsilon}(x)-V_{\delta_\epsilon}(y)\right| = \left|\int_{\mR^n}[\tilde{V}_0(x-z)-\tilde V_0(y-z)]\rho_{\delta_\epsilon}(z)\mathrm dz\right|\\
    &\le\int_{\mR^n}\left|\tilde{V}_0(x-z)-V_0(y-z)\right|\rho_{\delta_\epsilon}(z)\mathrm dz\\
    &\le\int_{\mR^n}(\beta_C-\epsilon_C)\|x-y\|_2\rho_{\delta_\epsilon}(z)\mathrm dz\\
    &=(\beta_C-\epsilon_C)\|x-y\|_2\underbrace{\int_{\mR^n}\rho_{\delta_\epsilon}(z)\mathrm dz}_{=1}.
    \end{align*}
    Moreover,
    it follows from Stone-Weierstrass theorem that there exists a $d_\epsilon$, such that for all $d>d_\epsilon$, $\|\nabla V_d-\nabla V_{\delta_\epsilon}\|_{\infty,X}\le\min\{\epsilon/(4\|T^*\|\|\mu\|_{\tv}\cdot\diam(X)),\epsilon_C,$ $\epsilon_1/(2\cdot\diam(X))\}$, where $V_d\in P_d$. 
And hence $\|V_d-V_{\delta_\epsilon}\|_{\infty,X}<\min\{\epsilon/(4\|T^*\|\|\mu\|_{\tv}),\epsilon_1/2\}$.     
    Therefore, it holds that $\|V_d- V_{0}\|_{\infty,X}\le\|V_d-V_{\delta_\epsilon}\|_{\infty,X}+\|V_{\delta_\epsilon}-\tilde V_{0}\|_{\infty,X} \le\min\{\epsilon/(2\|T^*\|\|\mu\|_{\tv}),\epsilon_1\}$ and $\|\nabla V_d\|_{\infty,X}\le \|\nabla V_{\delta_\epsilon}\|_{\infty,X}+\|\nabla V_d-\nabla V_{\delta_\epsilon}\|_{\infty,X}\le \beta_C-\epsilon_C+\min\{\epsilon,\epsilon_C\}\le\beta_C$.
    This implies that $(\theta_0,V_d)$ is also feasible to both \eqref{eq:consistencyOriginal} and \eqref{eq:consistencyApproxStep2}.
    Therefore, it holds that
    \begin{equation*}
        \begin{aligned}
            &\langle\theta_{\ell, d}^{\star T}\varphi \!+\! T(V_d^\star),\mu \rangle \leq \langle\theta_0^{T}\varphi \!+\! T(V_d),\mu \rangle\\
            &= \langle\theta_0^{T}\varphi \!+\! T(V_0),\mu \rangle + \langle V_d-V_0, T^*\mu\rangle\\
            &\le  \langle\theta_0^{T}\varphi \!+\! T(V_0),\mu \rangle + \underbrace{\|V_d-V_0\|_\infty}_{\le \epsilon/(2\|T^*\|\|\mu\|_{\tv})}\|T^*\|\|\mu\|_{\tv}\\
            &\leq\! \langle\theta_{\ell}^{\star T}\varphi \!+\! T(V^\star),\mu \rangle \!+\! \underbrace{\langle (\theta_0 \!-\! \theta_\ell^\star)^T\varphi \!+\! T(V_0\!-\!V^\star), \mu \rangle}_{\le \epsilon/2} \!+\! \epsilon/2\\
            &\leq \langle\theta_{\ell}^{\star T}\varphi \!+\! T(V^\star),\mu \rangle + \epsilon.
        \end{aligned}
     \end{equation*}

    This implies that for any $\epsilon>0$, a sufficiently high approximation order $d_\epsilon$ guarantees an $\epsilon$-close optimal solution. Consequently, the optimal value to the approximated problem converges to the optimal value of the original problem as the polynomial approximation order tends to infinity.
\end{pf}

Next, Lemma \ref{lem: Approximating phi by polynomials} shows that approximating the basis functions by polynomials also makes sense.

\begin{lemma}\label{lem: Approximating phi by polynomials}
    For $\phi:\mK\rightarrow R^n$ that is continuous and $\mu\in\mathcal{M}_+$, consider the optimization problem
    \begin{subequations}\label{eq:consistencyOriginalFiniteDim1}
    \begin{align}
    \min_{\substack{\theta\in\mR^n}} 
    &\;\langle\theta^T\phi,\mu \rangle\\
    \st &\;\theta^T\phi\!\ge\! 0,\forall \eta\in\mK,\\
    &\;\left\|\theta^T\phi\right\|_1 \geq 1,\\
    &\;\left\|\theta\right\|_1\leq \beta,
    \end{align}
    \end{subequations}
    as well as the optimization problem obtained by replacing the basis functions $\phi$ by a polynomial approximation $\tilde{{\phi}}_d$ of order $d$, that is, the approximated problem
    \begin{subequations}\label{eq:consistencyApproxStep1}
    \begin{align}
    \min_{\substack{\theta\in\mR^n}} 
    &\;\langle\theta^T\tilde{{\phi}}_d,\mu \rangle\\
    \st &\;\theta^T\tilde{{\phi}}_d\!\ge\! 0,\forall \eta\in\mK,\label{eq:nonnegative_cstr}\\
    &\;\left\|\theta^T\tilde{{\phi}}_d\right\|_1 \geq 1,\label{eq:1-norm_regularity_cstr}\\
    &\;\left\|\theta\right\|_1\leq \beta/2.\label{eq:theta_norm_cstr}
    \end{align}
    \end{subequations}
    Let $\theta^\star$ and $\theta_{d}^\star $ be the optimal solution to \eqref{eq:consistencyOriginalFiniteDim1} and \eqref{eq:consistencyApproxStep1} respectively. Suppose $\left\|\theta^\star\right\|_1 \leq \beta/4$ and the first elements in $\phi$ and $\tilde{\phi}_d$ are both $1$. 
    If $\lim_{d\rightarrow\infty}\left\|\phi - \tilde{\phi}_d\right\|_\infty = 0$, then there exists $\theta_{d}^\star +a_d$ that is feasible for \eqref{eq:consistencyOriginalFiniteDim1} with $\lim_{d\rightarrow\infty}\left\|a_d\right\|_1 = 0$ and $\lim_{d\rightarrow\infty}\langle\theta_{d}^{\star T}\phi,\mu \rangle = \langle\theta^{\star T}{\phi} ,\mu \rangle$.
\end{lemma}
\begin{pf}  
Since $\lim_{d \to \infty} \|\phi - \tilde{\phi}_d\|_\infty = 0$, for any small enough $\epsilon<1$, there exists a $d_\epsilon$ such that $\|\phi - \tilde{\phi}_d\|_\infty \leq \epsilon$ for every $d \geq d_\epsilon$.
This further implies that for all $\eta \in \mathbb{K}$,
\begin{align}
&| \theta^T \phi(\eta)\! -\! \theta^T \tilde{\phi}_d(\eta)|\! \le\!\! \sum_{i=1}^n\!|\theta_i|\underbrace{|\phi_i(\eta)\!-\!\tilde\phi_{d,i}(\eta)|}_{\le \epsilon}\!\le\! \|\theta\|_1 \epsilon,
\label{eq:abs_value_ineq}
\end{align}
where $\theta_i$, $\phi_i$ and $\tilde{\phi}_{d,i}$ are the $i$-th element of the corresponding vectors, respectively. 
Moreover, let $a_d= [\|\theta_{d}^\star \|_1 \epsilon, 0, \ldots, 0]^T$. Since $\theta_{d}^\star $ minimizes the approximated problem \eqref{eq:consistencyApproxStep1}, it must satisfy the constraints \eqref{eq:nonnegative_cstr}, \eqref{eq:1-norm_regularity_cstr} and \eqref{eq:theta_norm_cstr}; and since we assume the first elements of both $\phi$ and $\tilde \phi_d$ are both 1, 
it follows from \eqref{eq:abs_value_ineq} that
$(\theta_{d}^\star  + a_d)^T \phi(\eta) = \theta_{d}^{\star T} \phi(\eta)+\|\theta_{d}^\star  \|_1\epsilon\ge \theta_{d}^{\star T}\tilde{\phi}_d>0$.

Hence it further holds that
\begin{equation*}
    \begin{aligned}
     & \|(\theta_{d}^\star  \!+\! a_d)^T \phi\|_1\! =\!\! \int_{\mathbb{K}} (\theta_{d}^\star  \!+\! a_d)^T \phi(\eta) \mathrm d\eta
       \! \geq \!\! \int_{\mathbb{K}} \theta_{d}^{\star T} \tilde{\phi}_d(\eta) \mathrm d\eta\\
      &=\|\theta_d^\star\tilde\phi_d\|_1 \geq 1.
    \end{aligned}
\end{equation*}
Furthermore, $\theta_{d}^\star $ also must satisfy \eqref{eq:theta_norm_cstr} and hence
$ \| \theta_{d}^\star  + a_d \|_1 \le \| \theta_{d}^\star  \|_1 +\|a_d\|_1= \| \theta_{d}^\star  \|_1 +  \|\theta_d^\star\|_1 \epsilon \leq \beta$.
Thus $\theta_{d}^\star  + a_d$ is also feasible to the original problem \eqref{eq:consistencyOriginalFiniteDim1}, and hence it holds that $\langle \theta^{\star T} \phi, \mu \rangle\leq \langle (\theta_{d}^\star +a_d)^T \phi, \mu \rangle$ by optimality.

Similarly, let $b = [ \|\theta^\star\|_1 \epsilon, 0, \ldots, 0]^T$ and in view of we assume $\|\theta^\star\|_1\le \beta/4$, it follows that $\theta^\star + b$ is also feasible to the approximated problem \eqref{eq:consistencyApproxStep1} and it holds that $\langle \theta_{d}^{\star T} \tilde{\phi}_d, \mu \rangle  \leq \langle (\theta^\star+b)^T \tilde{\phi}_d, \mu \rangle$ by optimality.
Notably, since $\mu$ is a positive measure and in view of \eqref{eq:abs_value_ineq}, 
we can bound the objective function values by
\begin{equation*}
    \begin{aligned}
        &\langle (\theta_{d}^\star  + a_d)^T \phi, \mu \rangle = \langle \theta_{d}^{\star T} \phi, \mu \rangle + \langle \left\|\theta_{d}^\star \right\|_1\epsilon, \mu \rangle \\
        &\le \langle \theta_{d}^{\star T} \phi, \mu \rangle \!+\! 2\langle \underbrace{\left\|\theta_{d}^\star \right\|_1}_{\le \beta/2}\epsilon, \mu \rangle\leq \langle \theta_{d}^{\star T} \tilde{\phi}_d, \mu \rangle \!+\! \epsilon\langle \beta, \mu \rangle,
    \end{aligned}
\end{equation*}
and similarly, it holds that $\langle (\theta^\star + b)^T \tilde{\phi}_d, \mu \rangle \leq \langle \theta^{\star T} \phi, \mu \rangle + \epsilon\langle \beta/2, \mu \rangle$.
Therefore, it follows that
\begin{align*}
    &\langle\theta_{d}^{\star T}\phi,\mu \rangle = \langle (\theta_{d}^\star  + a_d)^T \phi, \mu \rangle - \langle \|\theta_{d}^\star \|_1 \epsilon, \mu \rangle\\
    &\;\geq \langle \theta^{\star T} \phi, \mu \rangle - \langle \|\theta_{d}^\star \|_1 \epsilon, \mu \rangle
    \geq \langle \theta^{\star T} \phi, \mu \rangle - \epsilon\langle \beta/2, \mu \rangle,\\
    &\langle\theta_{d}^{\star T}\phi,\mu \rangle\!\le\!\langle (\theta_{d}^\star  + a_d)^T \phi, \mu \rangle\!\leq\!\langle \theta_{d}^{\star T} \tilde{\phi}_d, \mu \rangle + \epsilon\langle \beta, \mu \rangle\\
    &\;\leq \langle (\theta^\star\! +\! b)^T \tilde{\phi}_d, \mu \rangle \!+\! \epsilon\langle \beta, \mu \rangle\!\leq\! \langle \theta^{\star T} \phi, \mu \rangle \!+\! \epsilon\langle 3 \beta/2, \mu \rangle.
\end{align*}
Since $\epsilon$ can be made arbitrarily small, the claim follows.
\end{pf}
Now that we are ready to prove Theorem \ref{thm: MAIN-CONSISTENCY-THEOREM}.
\begin{proof}[Proof of Theorem~4.2]
Let $(\theta_\ell,V,\psi)$ be any feasible solution to \eqref{eq:IOC_approx1}. 
By Assumption~\ref{ass:transition_kernel} and \ref{ass:obj_func}, and since constraints \eqref{eq:IOC_approx1_constraints} holds, it holds that for any $\eta_1,\eta_2$
\begin{align*}
&|\ell(\eta_1)\!+\!\int V(\cdot) Q(\cdot|\eta_1)\!-\!\ell(\eta_2)\!-\!\int V(\cdot) Q(\cdot|\eta_2)|\\
&\le|\ell(\eta_1)-\ell(\eta_2)|+|\int V(\cdot) Q(\cdot|\eta_1)-\int V(\cdot) Q(\cdot|\eta_2)|\\
&\le \left(\|\theta_\ell\|_1 L_\varphi+\|V\|_\infty L_Q\right)\|\eta_1-\eta_2\|_2,\\
&\le \left(\beta_\ell L_\varphi+\beta_V L_Q\right)\|\eta_1-\eta_2\|_2.
\end{align*}
Hence $\psi\in\lip(\beta_\ell L_\varphi+\beta_V L_Q)$, therefore $\psi$ lives in an equicontinuous set and Lemma~\ref{lem: stochastic-consistency-for-infinite-dim} applies. 
Therefore there exists $M_1'(\epsilon)$, such that for all $M\geq M_1'(\epsilon)$ and every feasible $\psi$,
\begin{align}
    \left|\langle \psi, \hat{\mu}^{\pitrue,N} \rangle- \frac{1}{M}\sum_{i=1}^M\sum_{t=1}^N \psi(\bm\eta_t^i) \right| \leq \epsilon,\quad\text{a.s.},
    \label{eq:psi_ulln}
\end{align}
where $\bm\eta_t^i:=(\mfx_t^i,\mfa_t^i)$.
Moreover, let $(\theta_{\ell,M}, V_M)$ be optimal to the finite-sample problem
\begin{equation}
\begin{aligned}
\min_{\theta_{\ell},V\in\mathcal{C}(X)} &\;\;\frac{1}{M}\sum_{i=1}^M\!\sum_{t=1}^N\psi(\bm\eta_t^i;\theta_\ell,V)\\
\st &\;\; \eqref{eq:IOC_approx1_constraints}.
\end{aligned}
\label{eq:data_psi_ioc_approx}
\end{equation}
By the second conclusion in Lemma~\ref{lem: stochastic-consistency-for-infinite-dim}, 
since the optimal value of \eqref{eq:IOC_approx1} is zero (cf. Proposition~\ref{prop:finite_time_horizon_equivalent}), we have
\begin{align}
    \langle \psi(\cdot;\theta_{\ell,M}, V_M),\!\hat\mu^{\pitrue,N} \rangle\leq \epsilon, \; \text{a.s.}\label{eq: M approx}
\end{align}

Next, recall that $(\thetatrue, \Vtrue)$ (and the corresponding $\bar\psi$) is an optimal solution of \eqref{eq:IOC_approx1}, and is therefore feasible for \eqref{eq:data_psi_ioc_approx}.
Since the bounds $\beta_\ell$, $\beta_V$ ($\beta_\psi$) can be chosen aribrarily large, we can construct a strictly feasible point by adding a positive constant to $\Vtrue$. 
Then by Lemma~\ref{lem: Approximating V by polynomials}, there exists $d_V'(\epsilon,M)$, such that for any $d_V\geq d_V'(\epsilon,M)$,
\begin{align}
    \left| \!\frac{1}{M}\sum_{i=1}^M\!\sum_{t=1}^N\left(\psi(\bm\eta_t^i;\theta_{\ell,M}^{d_V},\theta_{V,M}^{d_V,T}r)\!-\!\psi(\bm\eta_t^i;\theta_{\ell,M},V_M)\right)\!\right|\!\leq\!\epsilon,\label{eq: dv approx}
\end{align}
where $ \theta_{\ell,M}^{d_V},\theta_{V,M}^{d_V} $ is optimal to
\begin{equation}
\begin{aligned}
\min_{\theta_{\ell},\theta_V} &\;\;\frac{1}{M}\sum_{i=1}^M\!\sum_{t=1}^N\psi(\bm\eta_t^i;\theta_\ell,\theta_V^Tr)\\
\st &\;\; \psi(x,a;\theta_\ell,\theta_V^Tr)\!\ge\! 0,\forall (x,a)\!\in\!\mK,\\
&\;\;\int \psi(x,a;\theta_\ell,\theta_V^Tr) \mathrm dx\mathrm da\ge 1,\quad \eqref{eq:theta_V_ell_bounded}.
\end{aligned}
\label{eq:data_psi_f_ioc_approx}
\end{equation}

Finally, note that the continuous feature function vector $\varphi$ and cost-to-go $q^*V$ \eqref{eq:cost_to_go} can be uniformly approximated to arbitrary accuracy by polynomials on $\mathbb{K}$. 
Moreover, also note that we may, without loss of generality, assume the first element in $\varphi$ is one (this is the same as adding a constant in the cost function and does not change the optimal control policy), thus Lemma.\ref{lem: Approximating phi by polynomials} applies. 
Therefore, there exists $d_\psi'(\epsilon,M,d_V)$ such that for all $d_\psi\ge d_\psi'(\epsilon,M,d_V)$ and the optimal solution $(\theta_{\ell, M}^{\mathbf{d}}, \theta_{V, M}^{\mathbf{d}})$ of \eqref{eq:IOC_approx2}, it holds that
\begin{align}
    \left| \frac{1}{M}\!\sum_{i=1}^M\!\sum_{t=1}^N\!\!\left(\!\psi(\bm\eta_t^i;\theta_{\ell,M}^{\mathbf{d}},\theta_{V,M}^{\mathbf{d},T}r)\!-\!\psi(\bm\eta_t^i;\theta_{\ell,M}^{d_V},\theta_{V,M}^{d_V,T}r)\!\right)\!\right| \leq \epsilon.\label{eq: dpsi approx}
\end{align}
On the other hand, there also exists $a_{d_\psi}$ with $ \|a_{d_\psi}\|_1\leq \epsilon/\|\varphi\|_\infty $ such that
$\psi(\cdot;\theta_{\ell,M}^{\mathbf{d}}+a_{d_\psi},\theta_{V,M}^{\mathbf{d},T}r)$
is feasible for \eqref{eq:data_psi_f_ioc_approx}.
Thus $\psi(\eta;\theta_{\ell,M}^{\mathbf{d}}+a_{d_\psi},\theta_{V,M}^{\mathbf{d},T}r) \geq 0,\forall \eta\in\mK$, and hence
\begin{align*}
    &\langle \psi(\bm\eta_t^i;\theta_{\ell,M}^{\mathbf{d}},\theta_{V,M}^{\mathbf{d},T}r), \hat{\mu}^{\pitrue,N} \rangle \\
    &= \langle \psi(\bm\eta_t^i;\theta_{\ell,M}^{\mathbf{d}}+a_{d_\psi},\theta_{V,M}^{\mathbf{d},T}r), \hat{\mu}^{\pitrue,N} \rangle - \langle a_{d_\psi}^T\varphi, \hat{\mu}^{\pitrue,N} \rangle \\
    &\ge -\langle a_{d_\psi}^T\varphi, \hat{\mu}^{\pitrue,N} \rangle \ge -\|a_{d_\psi}\|_1\|\varphi\|_\infty\underbrace{\|\hat{\mu}^{\pitrue,N}\|_{\tv}}_{=N} \geq -N\epsilon.
\end{align*}

Now, for all $ M\geq M_1'(\epsilon), d_V\geq d_V'(\epsilon, M), d_\psi\geq d_\psi'(\epsilon,M,d_V)$, the following holds a.s. 
\begin{align*}
    &\langle \psi(\bm\eta_t^i;\theta_{\ell,M}^{\mathbf{d}},\theta_{V,M}^{\mathbf{d},T}r), \hat{\mu}^{\pitrue,N} \rangle \\
    &\overset{\eqref{eq:psi_ulln}}{\le} \frac{1}{M}\!\sum_{i=1}^M\!\sum_{t=1}^N\!\psi(\bm\eta_t^i;\theta_{\ell,M}^{\mathbf{d}},\theta_{V,M}^{\mathbf{d},T}r) + \epsilon \\
    &\overset{\eqref{eq: dpsi approx}}{\le} \frac{1}{M}\!\sum_{i=1}^M\!\sum_{t=1}^N\!\psi(\bm\eta_t^i;\theta_{\ell,M}^{d_V},\theta_{V,M}^{d_V,T}r) +2\epsilon \\
    &\overset{\eqref{eq: dv approx}}{\le} \frac{1}{M}\!\sum_{i=1}^M\!\sum_{t=1}^N\!\psi(\bm\eta_t^i;\theta_{\ell,M},V_M) + 3\epsilon\\
    &\overset{\eqref{eq:psi_ulln}}{\le} \langle \psi(\cdot;\theta_{\ell,M}, V_M),\!\hat\mu^{\pitrue,N} \rangle + 4\epsilon\overset{\eqref{eq: M approx}}{\le} 5\epsilon.
\end{align*}
and the claim follows.
\end{proof}
Furthermore, if the optimal solution $\theta_\ell^\star = \thetatrue$ is unique up to scalar normalization, namely, all optimal solutions differ only by a positive scaling factor, then the normalized approximated solution also converges to the normalized true optimum.
Otherwise, the true cost function parameters $\thetatrue$ is not identifiable under the current IOC problem structure. 
Nevertheless, the expert control policy can still be reconstructed by solving \eqref{eq:policy_approx_reconstruction}. The subsequent corollary establishes that such reconstruction error is also bounded.

\begin{corollary}
Suppose $\pi(x;\cdot):\Theta\rightarrow A$ is continuous for any $x$, and $\Theta$ is compact. 
    Let $\hat\pi_{M,\mathbf{d}}$ be the solution of \eqref{eq:policy_approx_reconstruction}, and suppose that there exists a $\theta_{\pitrue}$ such that $\pitrue=\pi(x;\theta_{\pitrue})$. It holds that
    \begin{align*}
        \lim_{M\rightarrow\infty}\lim_{d_V\rightarrow\infty}\lim_{d_\psi\rightarrow\infty}\langle \|\hat\pi_{M,\mathbf{d}}-\pitrue\|_2^2, \proj\hat\mu^{\pitrue,N} \rangle = 0,\;\text{a.s.}
    \end{align*}
\end{corollary}
\begin{pf}
    Since $\pi(\cdot;\theta_\pi)$ is continuous with respect to $\theta_\pi$ and $\theta_\pi$ belongs to a compact set, $\|\pi(x;\theta_\pi)-\pitrue(x)\|_2^2$ is continuous with respect to $\theta_\pi$ for each $x$. 
    Furthermore, since $\|\pi(\cdot;\theta_\pi)-\pitrue(\cdot)\|_2^2\leq\diam(A)^2$, the function class $\{\|\pi(\cdot;\theta_\pi)-\pitrue(\cdot)\|_2^2 \mid \theta_\pi\in \Theta\}$ is Glivenko-Cantelli \citep[E.g.~19.8, p.~272]{van1998asymptotic}. Thus, there exists $M_2'(\epsilon)$, such that $\forall M\geq M_2'(\epsilon)$ and $\pi\in\{\pi(\cdot;\theta_\pi)|\theta_\pi\in \Theta\}$, the following inequality holds
    \begin{align}
       & |\langle (\|\pi(\cdot)\!-\!\pitrue(\cdot)\|_2^2, \proj\hat\mu^{\pitrue,N} \rangle\! -\! \frac{1}{M}\sum_{i=1}^M\sum_{t=1}^N \|\pi(\mfx_{t}^i)\!-\!\pitrue(\mfx_{t}^i)\|_2^2|\nonumber\\
       & \quad \leq \epsilon/4,\;  \mbox{a.s.}
        \label{eq:pi_ulln}
    \end{align}

    On the other hand, let $(\theta_{\ell, M}^{\mathbf{d}}, \theta_{V, M}^{\mathbf{d}})$ be optimal to \eqref{eq:IOC_approx2}. 
    As shown in Theorem~\ref{thm: MAIN-CONSISTENCY-THEOREM}, for any $\epsilon >0$, there exists $M_3'(\epsilon), d_V'(\epsilon, M), d_{\psi,1}'(\epsilon, M, d_V)$, such that for all $M\geq M_3'(\epsilon), d_V\geq d_V'(\epsilon, M), d_\psi \geq d_{\psi,1}'(\epsilon, M, d_V)$, it holds $\langle\hat\psi_{M,\mathbf{d}}, \hat{\mu}^{\pitrue,N} \rangle \leq \epsilon/4$ a.s, 
    where $\hat\psi_{M,\mathbf{d}}:=\theta_{\ell, M}^{\mathbf{d},T}\varphi\!+\!(\alpha q^*\!-\!\proj^*)(\theta_{V, M}^{\mathbf{d},T}r)$.

    And applying Lemma~\ref{lem: Approximating phi by polynomials}, there also exists $d_{\psi,2}'(\epsilon, M, d_V)$ such that for all $d_\psi \geq d_{\psi,2}'(\epsilon, M, d_V)$, $\tilde{\psi}_{M,\mathbf{d}}:=(\theta_{\ell, M}^{\mathbf{d}}+a_{d_\psi})^T\varphi\!+\!(\alpha q^*\!-\!\proj^*)(\theta_{V, M}^{\mathbf{d},T}r)$ is feasible for \eqref{eq:bellman_inequality_constraint} with $\|a_{d_\psi}\|_1\leq \epsilon/( 4N\| \varphi \|_\infty)$.
    Thus we have $ \hat\psi_{M,\mathbf{d}} (x,a)\geq \tilde{\psi}_{M,\mathbf{d}}(x,a) - \|a_{d_\psi}\|_1\| \varphi \|_\infty \geq -\epsilon/(4N), \forall (x,a)\in\mK$. 
    Therefore, since $\hat\pi_{M,\mathbf{d}}$ is optimal to \eqref{eq:policy_approx_reconstruction}, and in view of \eqref{eq:psi_ulln}, it follows that
    \begin{align*}
        &\frac{1}{M}\sum_{i=1}^M\sum_{t=1}^N \left(-\epsilon/(4N) + \|\hat\pi_{M,\mathbf{d}}(\mfx_{t}^i)-\pitrue(\mfx_{t}^i)\|_2^2\right)  \\
        &\leq \frac{1}{M}\sum_{i=1}^M\sum_{t=1}^N \hat\psi_{M,\mathbf{d}}(\mfx_{t}^i, \hat\pi_{M,\mathbf{d}}(\mfx_{t}^i)) + \|\hat\pi_{M,\mathbf{d}}(\mfx_{t}^i)-\pitrue(\mfx_{t}^i)\|_2^2\\
        &\leq \frac{1}{M}\sum_{i=1}^M\sum_{t=1}^N \hat\psi_{M,\mathbf{d}}(\mfx_{t}^i, \pitrue(\mfx_{t}^i)) + \|\pitrue(\mfx_{t}^i)-\pitrue(\mfx_{t}^i)\|_2^2 \\
        &=\! \frac{1}{M}\sum_{i=1}^M\!\!\sum_{t=1}^N\hat\psi_{M,\mathbf{d}}(\mfx_{t}^i, \mfa_{t}^i)\!\leq\! \underbrace{\langle \hat\psi_{M,\mathbf{d}}, \hat{\mu}^{\pitrue,N} \rangle}_{\le \epsilon/4}\! +\! \epsilon/4 \!\leq\!  \epsilon/2,\;\mbox{a.s.}.
    \end{align*}
    Thus, it follows from \eqref{eq:pi_ulln} that $\forall M\!\ge\!\max\{M_1'(\epsilon),\!M_2'(\epsilon),\!$ $M_3'(\epsilon)\}$, $d_V\geq d_V'(\epsilon, M), d_\psi \geq \max\{d_{\psi,1}'(\epsilon, M, d_V),$ $d_{\psi,2}'(\epsilon, M, d_V)\}$, 
    it holds that $\langle \|\hat\pi_{M,\mathbf{d}}(\cdot)-\pitrue(\cdot)\|_2^2, \proj\hat\mu^{\pitrue,N} \rangle \leq \frac{1}{M}\sum_{i=1}^M\sum_{t=1}^N $ 
    $\|\hat\pi_{M,\mathbf{d}}(\mfx_{t}^i)-\pitrue(\mfx_{t}^i)\|_2^2 + \epsilon/4 \leq \epsilon$ a.s.
    And hence the statement follows.

\end{pf}

%% file: 6-PraticalImplement.tex
As demonstrated in Theorem \ref{thm: MAIN-CONSISTENCY-THEOREM}, the IOC problem can be solved numerically by solving \eqref{eq:IOC_approx2}, where $\hat{\psi} = \theta_\psi^T\phi\in \mathscr{Q}(g_{1:n_x+n_a})$. 
This constraint corresponds to a Weighted Sum-of-Squares (WSOS) polynomial condition.
Consequently, we solve such problem by following the pipeline proposed in \citep{Papp2019sum}, which treats the problem as a non-symmetric conic optimization problem. In particular, it constructs a logarithmically homogeneous self-concordant barrier (LHSCB) function, and applies a primal-dual interior point method. 
This framework is significantly more efficient and numerically stable than classical semidefinite-programming-based approaches.
Next, we will explain the procedure briefly for self-containment.

As mentioned in Sec.~\ref{sec:finite-dim-ioc}, we represent the polynomials by linear combination of Lagrange interpolation basis. This turns the WSOS polynomial constraints to conic constraints on polynomial coefficients.  
Next, we reformulate \eqref{eq:IOC_approx2} to the following form
\begin{equation}\label{eq:dual-standform}
    \begin{aligned}
    \max_{\theta_\psi, \theta_\ell, \theta_V} &(-\Xi_\psi^Th)^T\begin{bmatrix}\theta_\ell\\\theta_V\end{bmatrix},\\
    \text{s.t. }& \begin{bmatrix}d^T(-\Xi_\psi)\\ -\Xi_\psi  \end{bmatrix}\begin{bmatrix}\theta_\ell\\\theta_V\end{bmatrix}
    + \begin{bmatrix} s_1\\\theta_\psi\end{bmatrix}
    = \begin{bmatrix}-1\\0\end{bmatrix},\\
    & \begin{bmatrix}s_1\\\theta_\psi\end{bmatrix} \in\mR_+\times\mathcal{K},
\end{aligned}
\end{equation}
where $\Xi_\psi$ is defined in \eqref{eq:psi_ell_V_equality_constraint}, $\mathcal{K} := \left\{\theta\in\mR^{D_\psi} \mid \theta^T\phi\in \right.$ $\left.\mathscr{Q}(g_{1:n_x+n_a})\right\}$ is the conic constraints on the polynomial coefficients, and $s_1\in\mR_+$ is a slack variable. 
Moreover, as also mentioned in Sec.~\ref{sec:finite-dim-ioc}, in practice, we can choose the norm bounds $\beta_\psi^\prime$, $\beta_\ell^\prime$, $\beta_V^\prime$ in \eqref{eq:IOC_approx2} to be arbitrarily large to include the ``true" cost function $\bar{\ell}$ and the ``true'' value function $\bar{V}$ as an interior point of the norm constraints \eqref{eq:theta_psi_bounded}, \eqref{eq:theta_V_ell_bounded}.
This means that any ``reasonable" estimator should make the norm bound constraints \eqref{eq:theta_psi_bounded} and \eqref{eq:theta_V_ell_bounded} inactive, otherwise it means that the orders of the polynomials or the amount of the data are just not sufficient.
Therefore, in practice, we can safely ignore \eqref{eq:theta_psi_bounded} and \eqref{eq:theta_V_ell_bounded} and if the numerical solver fails to solve \eqref{eq:dual-standform}, it implies that we need to add more data or increase the orders of the polynomials.

Next, we view \eqref{eq:dual-standform} as the dual problem of a non-symmetric conic optimization problem, which takes the form
\begin{equation}
\begin{aligned}
    \min_{y\in \mR^{D_\psi+1}} &-y_1,\\
    \mathrm{s.t. }&\begin{bmatrix}\Xi_\psi^Td &\Xi_\psi^T  \end{bmatrix}y =\Xi_\psi^Th,\\
    & y\in \mR_+\times\mathcal{K}^*,
\end{aligned}
\label{eq:final_primal_problem}
\end{equation}
where $\mathcal{K}^*$ is the dual cone of $\mathcal{K}$.
To handle the cone $\mathcal{K}^*$, we adopt the LHSCB function $F:(\mathcal{K}^*)^\circ\rightarrow\mR$ given by \citep[Sec.~6]{Papp2019sum}, where $(\mathcal{K}^*)^\circ$ is the interior of $\mathcal{K}^*$. 
More specifically, let $q$ and $p$ denote the polynomial bases of degree $2d_\psi$ and $(d_\psi-1)$, respectively. The barrier function $F$ shall take the form
\begin{align*}
    F(y) = -\sum_{i=1}^{n_x+n_a}\mathrm{ln}\left(\mathrm{det}\left(\Lambda_i(y)\right)\right),
\end{align*}
where $\Lambda_i$ is the linear operator uniquely defined by $\Lambda_i(q)=g_i pp^T$.
Note that the operator $\Lambda_i$ depends explicitly on the choice of the polynomial basis, and, as mentioned in Remark~\ref{remark:using-interpolation}, we employ the Lagrange interpolation basis rather than the monomial basis since the latter often yields ill-conditioned $\Lambda_i(y)$ and degrades numerical performance. 
Moreover, we also use the approximated Fekete points \citep{Sommariva2009FeketePoints} as the Lagrange interpolation nodes to avoid making the resulting polynomial sensitive to the interpolation values.

Finally, we solve \eqref{eq:final_primal_problem} with the primal-dual interior-point method proposed in \citep{Skajaa2015pdinterior}. All subsequent numerical experiments are conducted based on this implementation.

%% file: 8-Conclusion.tex
In this paper, we investigate the IOC problem for discrete-time, nonlinear stochastic systems with infinite-horizon discounted cost. 
In particular, we propose a polynomial-approximation-based approach that reformulates the IOC problem into a sum-of-square optimization problem. This convex formulation avoids local minima and maintains theoretical interpretability.

The proposed estimator is shown to be asymptotically and statistically consistent: as the polynomial approximation order and the number of demonstration trajectories increase, the expert policy becomes nearly optimal under the estimated cost function, and the corresponding optimality gap asymptotically approaches zero.
Numerical experiments on three typical optimal control systems validate the theoretical consistency and further demonstrate the estimation accuracy and generalization capability of the proposed method.

Future work will extend the proposed framework to higher-dimensional systems, addressing the scalability challenges of polynomial and SOS approximations.
Another promising direction is to handle information-limited settings, including partially observed states, unknown dynamics, or unmeasured control inputs.

%% file: manuscript_v1.bbl
\begin{thebibliography}{10}

\bibitem{abbeel2004apprenticeship}
Pieter Abbeel and Andrew~Y Ng.
\newblock Apprenticeship learning via inverse reinforcement learning.
\newblock In {\em Proceedings of the twenty-first international conference on
  Machine learning}, page~1, 2004.

\bibitem{Ashwood2022Dynamic}
Zoe Ashwood, Aditi Jha, and Jonathan~W Pillow.
\newblock Dynamic inverse reinforcement learning for characterizing animal
  behavior.
\newblock In {\em Advances in Neural Information Processing Systems},
  volume~35, pages 29663--29676, 2022.

\bibitem{codevilla2019exploring}
Felipe Codevilla, Eder Santana, Antonio~M L{\'o}pez, and Adrien Gaidon.
\newblock Exploring the limitations of behavior cloning for autonomous driving.
\newblock In {\em Proceedings of the IEEE/CVF international conference on
  computer vision}, pages 9329--9338, 2019.

\bibitem{dey2023inverse}
Sourav Dey, Thibault Marzullo, and Gregor Henze.
\newblock Inverse reinforcement learning control for building energy
  management.
\newblock {\em Energy and Buildings}, 286:112941, 2023.

\bibitem{dvijotham2010inverse}
Krishnamurthy Dvijotham and Emanuel Todorov.
\newblock Inverse optimal control with linearly-solvable mdps.
\newblock In {\em Proceedings of the 27th International conference on machine
  learning (ICML-10)}, pages 335--342, 2010.

\bibitem{fernandez2025estimating}
Victor~Nan Fernandez-Ayala, Shankar~A Deka, and Dimos~V Dimarogonas.
\newblock Estimating unknown dynamics and cost as a bilinear system with
  koopman-based inverse optimal control.
\newblock {\em arXiv preprint arXiv:2501.18318}, 2025.

\bibitem{pmlr-v48-finn16}
Chelsea Finn, Sergey Levine, and Pieter Abbeel.
\newblock Guided cost learning: Deep inverse optimal control via policy
  optimization.
\newblock In Maria~Florina Balcan and Kilian~Q. Weinberger, editors, {\em
  Proceedings of The 33rd International Conference on Machine Learning},
  volume~48 of {\em Proceedings of Machine Learning Research}, pages 49--58,
  New York, New York, USA, 20--22 Jun 2016. PMLR.

\bibitem{fu2017learning}
Justin Fu, Katie Luo, and Sergey Levine.
\newblock Learning robust rewards with adversarial inverse reinforcement
  learning.
\newblock In {\em International Conference on Learning Representations}, 2018.

\bibitem{garrabe2025convex}
Emiland Garrabe, Hozefa Jesawada, Carmen Del~Vecchio, and Giovanni Russo.
\newblock On convex data-driven inverse optimal control for nonlinear,
  non-stationary and stochastic systems.
\newblock {\em Automatica}, 173:112015, 2025.

\bibitem{Guo2021Deep}
Hongye Guo, Qixin Chen, Qing Xia, and Chongqing Kang.
\newblock Deep inverse reinforcement learning for objective function
  identification in bidding models.
\newblock {\em IEEE Transactions on Power Systems}, 36(6):5684--5696, 2021.

\bibitem{guo2023imitation}
Taosha Guo, Abed~AlRahman Al~Makdah, Vishaal Krishnan, and Fabio Pasqualetti.
\newblock Imitation and transfer learning for lqg control.
\newblock {\em IEEE Control Systems Letters}, 7:2149--2154, 2023.

\bibitem{hernandez2012discrete}
On{\'e}simo Hern{\'a}ndez-Lerma and Jean~Bernard Lasserre.
\newblock {\em Discrete-time Markov control processes: basic optimality
  criteria}.
\newblock Springer Science \& Business Media, 1996.

\bibitem{Ho2016GAIL}
Jonathan Ho and Stefano Ermon.
\newblock Generative adversarial imitation learning.
\newblock In {\em Advances in Neural Information Processing Systems},
  volume~29, 2016.

\bibitem{Jean2019Injectivity}
Frédéric Jean and Sofya Maslovskaya.
\newblock Injectivity of the inverse optimal control problem for control-affine
  systems.
\newblock In {\em 2019 IEEE 58th Conference on Decision and Control (CDC)},
  pages 511--516, 2019.

\bibitem{kallenberg1997foundations}
Olav Kallenberg.
\newblock {\em Foundations of modern probability}.
\newblock Springer, 1997.

\bibitem{kreyszig1991introductory}
Erwin Kreyszig.
\newblock {\em Introductory functional analysis with applications}.
\newblock John Wiley \& Sons, 1991.

\bibitem{laurent2009sums}
Monique Laurent.
\newblock Sums of squares, moment matrices and optimization over polynomials.
\newblock In {\em Emerging applications of algebraic geometry}, pages 157--270.
  Springer, 2009.

\bibitem{li2018convex}
Yibei Li, Han Zhang, Yu~Yao, and Xiaoming Hu.
\newblock A convex optimization approach to inverse optimal control.
\newblock In {\em 2018 37th Chinese Control Conference (CCC)}, pages 257--262.
  IEEE, 2018.

\bibitem{lian2022inverse}
Bosen Lian, Wenqian Xue, Frank~L Lewis, and Tianyou Chai.
\newblock Inverse reinforcement learning for multi-player noncooperative
  apprentice games.
\newblock {\em Automatica}, 145:110524, 2022.

\bibitem{Liang2023DataDriven}
Zihao Liang, Wenjian Hao, and Shaoshuai Mou.
\newblock A data-driven approach for inverse optimal control.
\newblock In {\em 2023 62nd IEEE Conference on Decision and Control (CDC)},
  pages 3632--3637, 2023.

\bibitem{likmeta2021dealing}
Amarildo Likmeta, Alberto~Maria Metelli, Giorgia Ramponi, Andrea Tirinzoni,
  Matteo Giuliani, and Marcello Restelli.
\newblock Dealing with multiple experts and non-stationarity in inverse
  reinforcement learning: an application to real-life problems.
\newblock {\em Machine Learning}, 110(9):2541--2576, 2021.

\bibitem{mcshane1934extension}
Edward~James McShane.
\newblock Extension of range of functions.
\newblock 1934.

\bibitem{Mehr2023Maximum}
Negar Mehr, Mingyu Wang, Maulik Bhatt, and Mac Schwager.
\newblock Maximum-entropy multi-agent dynamic games: Forward and inverse
  solutions.
\newblock {\em IEEE Transactions on Robotics}, 39(3):1801--1815, 2023.

\bibitem{Menner2018Convex}
Marcel Menner and Melanie~N. Zeilinger.
\newblock Convex formulations and algebraic solutions for linear quadratic
  inverse optimal control problems.
\newblock In {\em 2018 European Control Conference (ECC)}, pages 2107--2112,
  2018.

\bibitem{mombaur2010human}
Katja Mombaur, Anh Truong, and Jean-Paul Laumond.
\newblock From human to humanoid locomotion -- an inverse optimal control
  approach.
\newblock {\em Autonomous robots}, 28(3):369--383, 2010.

\bibitem{Papp2019sum}
Dávid Papp and Sercan Yildiz.
\newblock Sum-of-squares optimization without semidefinite programming.
\newblock {\em SIAM Journal on Optimization}, 29(1):822--851, 2019.

\bibitem{pauwels2016linear}
Edouard Pauwels, Didier Henrion, and Jean-Bernard Lasserre.
\newblock Linear conic optimization for inverse optimal control.
\newblock {\em SIAM Journal on Control and Optimization}, 54(3):1798--1825,
  2016.

\bibitem{pmlr-v70-pinto17a}
Lerrel Pinto, James Davidson, Rahul Sukthankar, and Abhinav Gupta.
\newblock Robust adversarial reinforcement learning.
\newblock In Doina Precup and Yee~Whye Teh, editors, {\em Proceedings of the
  34th International Conference on Machine Learning}, volume~70 of {\em
  Proceedings of Machine Learning Research}, pages 2817--2826. PMLR, 06--11 Aug
  2017.

\bibitem{putinar1993positive}
Mihai Putinar.
\newblock Positive polynomials on compact semi-algebraic sets.
\newblock {\em Indiana University Mathematics Journal}, 42(3):969--984, 1993.

\bibitem{rickenbach2024inverse}
Rahel Rickenbach, Anna Scampicchio, and Melanie~N Zeilinger.
\newblock Inverse optimal control as an errors-in-variables problem.
\newblock In {\em 6th Annual Learning for Dynamics \& Control Conference},
  pages 375--386. PMLR, 2024.

\bibitem{Rodrigues2022Inverse}
Luis Rodrigues.
\newblock Inverse optimal control with discount factor for continuous and
  discrete-time control-affine systems and reinforcement learning.
\newblock In {\em 2022 IEEE 61st Conference on Decision and Control (CDC)},
  pages 5783--5788, 2022.

\bibitem{ross2011reduction}
St{\'e}phane Ross, Geoffrey Gordon, and Drew Bagnell.
\newblock A reduction of imitation learning and structured prediction to
  no-regret online learning.
\newblock In {\em Proceedings of the fourteenth international conference on
  artificial intelligence and statistics}, pages 627--635. JMLR Workshop and
  Conference Proceedings, 2011.

\bibitem{rudin1987real}
Walter Rudin.
\newblock {\em Real and complex analysis}.
\newblock McGraw-Hill, Inc., 1987.

\bibitem{sammut2017behavioral}
Claude Sammut.
\newblock Behavioral cloning.
\newblock In Claude Sammut and Geoffrey~I. Webb, editors, {\em Encyclopedia of
  Machine Learning and Data Mining}, pages 120--124. Springer, Boston, MA, USA,
  2017.

\bibitem{sauer1995multivariate}
Thomas Sauer and Yuan Xu.
\newblock On multivariate {L}agrange interpolation.
\newblock {\em Mathematics of computation}, 64(211):1147--1170, 1995.

\bibitem{savorgnan2009discrete}
Carlo Savorgnan, Jean~B Lasserre, and Moritz Diehl.
\newblock Discrete-time stochastic optimal control via occupation measures and
  moment relaxations.
\newblock In {\em Proceedings of the 48h IEEE Conference on Decision and
  Control (CDC) held jointly with 2009 28th Chinese Control Conference}, pages
  519--524. IEEE, 2009.

\bibitem{Schultheis2021Inverse}
Matthias Schultheis, Dominik Straub, and Constantin~A Rothkopf.
\newblock Inverse optimal control adapted to the noise characteristics of the
  human sensorimotor system.
\newblock In {\em Advances in Neural Information Processing Systems},
  volume~34, pages 9429--9442, 2021.

\bibitem{Skajaa2015pdinterior}
Anders Skajaa and Yinyu Ye.
\newblock A homogeneous interior-point algorithm for nonsymmetric convex conic
  optimization.
\newblock {\em Mathematical Programming}, 150(2):391--422, 2015.

\bibitem{Sommariva2009FeketePoints}
Alvise Sommariva and Marco Vianello.
\newblock Computing approximate {Fekete} points by {QR} factorizations of
  {Vandermonde} matrices.
\newblock {\em Computers \& Mathematics with Applications}, 57(8):1324--1336,
  2009.

\bibitem{torabi2018behavioralcloningobservation}
Faraz Torabi, Garrett Warnell, and Peter Stone.
\newblock Behavioral cloning from observation.
\newblock In {\em Proceedings of the 27th International Joint Conference on
  Artificial Intelligence}, page 4950–4957, 2018.

\bibitem{van1998asymptotic}
Adrianus~W. van~der Vaart.
\newblock {\em Asymptotic statistics}.
\newblock Cambridge university press, Cambridge, United Kingdom, 1998.

\bibitem{wellner2013weak}
Jon Wellner et~al.
\newblock {\em Weak convergence and empirical processes: with applications to
  statistics}.
\newblock Springer Science \& Business Media, 2013.

\bibitem{yu2021system}
Chengpu Yu, Yao Li, Hao Fang, and Jie Chen.
\newblock System identification approach for inverse optimal control of
  finite-horizon linear quadratic regulators.
\newblock {\em Automatica}, 129:109636, 2021.

\bibitem{zare2024survey}
Maryam Zare, Parham~M Kebria, Abbas Khosravi, and Saeid Nahavandi.
\newblock A survey of imitation learning: Algorithms, recent developments, and
  challenges.
\newblock {\em IEEE Transactions on Cybernetics}, 2024.

\bibitem{zhang2019inverseCDC}
Han Zhang, Yibei Li, and Xiaoming Hu.
\newblock Inverse optimal control for finite-horizon discrete-time linear
  quadratic regulator under noisy output.
\newblock In {\em 2019 IEEE 58th Conference on Decision and Control (CDC)},
  pages 6663--6668. IEEE, 2019.

\bibitem{zhang2021inverse}
Han Zhang and Axel Ringh.
\newblock Inverse linear-quadratic discrete-time finite-horizon optimal control
  for indistinguishable homogeneous agents: A convex optimization approach.
\newblock {\em Automatica}, 148:110758, 2023.

\bibitem{zhang2024inverse}
Han Zhang and Axel Ringh.
\newblock Inverse optimal control for averaged cost per stage linear quadratic
  regulators.
\newblock {\em Systems \& Control Letters}, 183:105658, 2024.

\bibitem{zhang2024statistically}
Han Zhang and Axel Ringh.
\newblock Statistically consistent inverse optimal control for discrete-time
  indefinite linear--quadratic systems.
\newblock {\em Automatica}, 166:111705, 2024.

\bibitem{zhang2022statistically}
Han Zhang, Axel Ringh, Weihan Jiang, Shaoyuan Li, and Xiaoming Hu.
\newblock Statistically consistent inverse optimal control for linear-quadratic
  tracking with random time horizon.
\newblock In {\em 2022 41st Chinese Control Conference (CCC)}, pages
  1515--1522, 2022.

\bibitem{zhang2019inverse}
Han Zhang, Jack Umenberger, and Xiaoming Hu.
\newblock Inverse optimal control for discrete-time finite-horizon linear
  quadratic regulators.
\newblock {\em Automatica}, 110:108593, 2019.

\bibitem{zhou2018Infinite}
Zhengyuan Zhou, Michael Bloem, and Nicholas Bambos.
\newblock Infinite time horizon maximum causal entropy inverse reinforcement
  learning.
\newblock {\em IEEE Transactions on Automatic Control}, 63(9):2787--2802, 2018.

\bibitem{ziebart2008maximum}
Brian~D Ziebart, Andrew~L Maas, J~Andrew Bagnell, Anind~K Dey, et~al.
\newblock Maximum entropy inverse reinforcement learning.
\newblock In {\em Aaai}, volume~8, pages 1433--1438. Chicago, IL, USA, 2008.

\end{thebibliography}
